\documentclass[10pt]{amsart}
\usepackage{verbatim}
\usepackage{color,amssymb}
\usepackage{geometry,esint}
\usepackage[colorlinks=true,urlcolor=blue, citecolor=red,linkcolor=blue,linktocpage,pdfpagelabels, bookmarksnumbered,bookmarksopen]{hyperref}
\usepackage[hyperpageref]{backref}

\newtheorem{teo}{Theorem}[section]
\newtheorem{lm}[teo]{Lemma}
\newtheorem*{main}{Main Theorem}
\newtheorem{prop}[teo]{Proposition}
\theoremstyle{definition}

\newtheorem{rem}[teo]{Remark}
\newtheorem*{ack}{Acknowledgments}

\numberwithin{equation}{section}

\title[Orthotropic nonlinear diffusion]{Gradient estimates for\\ an orthotropic nonlinear diffusion equation}
\author[Bousquet]{Pierre Bousquet}
\author[Brasco]{Lorenzo Brasco}
\author[Leone]{Chiara Leone}
\author[Verde]{Anna Verde}

\address[P. Bousquet]{Institut de Math\'ematiques de Toulouse, CNRS UMR 5219
\newline\indent Universit\'e de Toulouse
\newline\indent F-31062 Toulouse Cedex 9, France.}
\email{pierre.bousquet@math.univ-toulouse.fr}

\address[L.\ Brasco]{Dipartimento di Matematica e Informatica
	\newline\indent
	Universit\`a degli Studi di Ferrara
	\newline\indent
	Via Machiavelli 35, 44121 Ferrara, Italy}
\email{lorenzo.brasco@unife.it}

\address[C.\ Leone \& A.\ Verde]{Dipartimento di Matematica ``R. Caccioppoli''
\newline\indent
Universit\`a degli Studi di Napoli ``Federico II''
\newline\indent
Via Cinthia, Complesso Universitario di Monte S. Angelo, 80126 Napoli, Italy}
\email{chiara.leone@unina.it}
\email{anna.verde@unina.it}

\subjclass[2010]{35K65, 35B65, 35K92}
\keywords{Degenerate parabolic equations, Lipschitz continuity, anisotropic operators.}

\begin{document}

\begin{abstract}
We consider a quasilinear degenerate parabolic equation driven by the orthotropic $p-$Laplacian. We prove that local weak solutions are locally Lipschitz continuous in the spatial variable, uniformly in time.
\end{abstract}

\maketitle

\begin{center}
\begin{minipage}{10cm}
\small
\tableofcontents
\end{minipage}
\end{center}

\section{Introduction}

\subsection{Aim of the paper}
Let $\Omega\subset\mathbb{R}^N$ be an open bounded set and $I\subset\mathbb{R}$ an open bounded interval. We study the gradient regularity of local weak solutions to the following parabolic equation 
\begin{equation}
\label{equationL}
u_t=\sum_{i=1}^N \left(|u_{x_i}|^{p-2}\,u_{x_i}\right)_{x_i},\qquad \mbox{ in }I\times\Omega.
\end{equation}
Evolution equations of this type have been studied since the 60's of the previous century, especially by the Soviet school, see for example the paper \cite{Vi} by Vishik. Equation \eqref{equationL} also explicitely appears in the monographs \cite{Li}, \cite[Example 4.A, Chapter III]{Sh} and \cite[Example 30.8]{Zei}, among others. 
\par
In this paper, we will focus on the case $p\ge 2$. We first observe that \eqref{equationL} looks quite similar to the more familiar one
\begin{equation}
\label{equationP}
u_t=\Delta_p u,\qquad \mbox{ in }I\times\Omega,
\end{equation}
which involves the $p-$Laplace operator 
\[
\Delta_pu=\sum_{i=1}^N \left(|\nabla u|^{p-2}\,u_{x_i}\right)_{x_i}.
\]
Indeed, both parabolic equations are particular instances of equations of the type
\[
u_t=\mathrm{div}\nabla F(\nabla u)
\]
with $F:\mathbb{R}^N\to\mathbb{R}$ a convex function which satisfies the structural conditions
\[
\langle \nabla F(z),z\rangle\ge \frac{1}{C}\,|z|^p\qquad \mbox{ and }\qquad |\nabla F(z)|\le C\,|z|^{p-1},\qquad \mbox{ for every }z\in\mathbb{R}^N.
\]
Then, the basic regularity theory equally applies to both \eqref{equationL} and \eqref{equationP}. The standard reference in the field is DiBenedetto's monograph \cite{Dib}, where one can find 
boundedness results for the solution $u$ (see \cite[Chapter V]{Dib}), H\"older continuous estimates for $u$ (see \cite[Chapter III]{Dib}), as well as Harnack inequality for positive solutions (see \cite[Chapter VI]{Dib}).  At a technical level, there is no distinction to be made between \eqref{equationL} and \eqref{equationP}.
\vskip.2cm\noindent
In contrast,  when coming to the regularity of $\nabla u$ (i.e. boundedness and continuity), the situation  becomes fairly more complicated. Let us start from \eqref{equationP}. DiBenedetto and  Friedman  \cite{DibFR} have proved that the gradients of this equation are bounded. This is the starting point to obtain the continuity of the gradients for any 
\[
p>\frac{2\,N}{N+2}.
\]  
We refer again to DiBenedetto's book for a comprehensive collection of results on the subject, notably to \cite[Chapter VIII]{Dib}. 
Since then, there has been a growing literature concerning the regularity for nonlinear, possibly degenerate or singular, parabolic equations (or systems), the main model of which is given by the evolutionary $p-$Laplacian equation \eqref{equationP}. Without any attempt to completeness, we can just mention some classical references \cite{DibFR1, Dib1, Wi, Chedib88,  Choe1}, up  to the most recent contributions on the subject, given by \cite{BDMi, KM1, KM2}, among others. 
\par
However, none of these results apply to our equation \eqref{equationL}. Indeed, all of them rely on the fact that the loss of ellipticity of the operator $\mathrm{div}\nabla F$ is restricted to a single point, since the Hessian $D^2 F$ behaves as in the model case \eqref{equationP}
\[
\langle D^2 F(z)\,\xi,\xi\rangle \ge \frac{1}{C}\, |z|^{p-2}\,|\xi|^2,
\]
where the elliptic character is lost only for $z=0$.
Such a property dramatically breaks down for our equation \eqref{equationL}. Indeed, in this case, the function $F$ has the following {\it orthotropic structure}
\begin{equation}
\label{orto}
F(z)=\frac{1}{p}\,\sum_{i=1}^N |z_i|^p,\qquad \mbox{ for every } z\in\mathbb{R}^N.
\end{equation}
The Hessian matrix of \(F\) now   {\it degenerates on an unbounded set}, namely the set of those \(z\in \mathbb{R}^N\) such that one component $z_i$ is \(0\). As a consequence,  the aforementioned references do not provide any regularity results for the gradients of the solutions. 
\vskip.2cm\noindent
The main goal of the present paper is to prove the $L^\infty$ bound on $\nabla u$ for our equation \eqref{equationL}, thus extending the result by DiBenedetto and Friedman to this more degenerate setting. In order to do this, we will need to adapt to the parabolic setting the machinery that we developed in \cite{BouBra, BouBraJul, BouBraLeoVer, BraCar} and \cite{BLPV}, for degenerate {\it equations with orthotropic structure}. Indeed, the operator 
\[
\sum_{i=1}^N \left(|u_{x_i}|^{p-2}\,u_{x_i}\right)_{x_i},
\]
that we called {\it orthotropic $p-$Laplacian}, is the prominent example of this kind of equations. We also refer to \cite{D} for an approach to this operator, based on viscosity techniques.
\subsection{Main result}

In this paper, we establish the following regularity result which can be seen as the parabolic counterpart of our previous result \cite[Theorem 1.1]{BouBraLeoVer} for the elliptic case. In the statement below, the notation \(\nabla u\) refers to the spatial variables, i.\,e. \(\nabla u=(u_{x_1}, \dots, u_{x_N})\).

\begin{main}
Let $p>2$ and let $u\in L^p_{\rm loc}(I;W^{1,p}_{\rm loc}(\Omega))$ be a local weak solution of \eqref{equationL}. Then  $\nabla u\in L^\infty_{\rm loc}(I\times \Omega)$. More precisely, for every parabolic cube
\[
\mathcal{Q}_{\tau,R}(t_0,x_0):=(t_0-\tau,t_0)\times(x_0-R,x_0+R)^N\Subset I\times \Omega,
\]
and every $0<\sigma<1$, we have
\begin{equation}
\label{stimapriori}
\begin{split}
\|\nabla u\|_{L^\infty(\mathcal{Q}_{\sigma\tau,\sigma R}(t_0,x_0))}&\le C\,\frac{1}{(1-\sigma)^\frac{N+2}{2}}\,\left(\frac{\tau}{R^2}\right)^\frac{1}{2}\,\left(\fint_{\mathcal{Q}_{\tau,R}(t_0,x_0)} |\nabla u|^p\,dt\,dx\right)^\frac{1}{2}\\
&+C\,\left((1-\sigma)\,\frac{R^2}{\tau}\right)^\frac{1}{p-2},
\end{split}
\end{equation}
for a constant $C=C(N,p)>0$. 
\end{main}
\begin{rem}[Scalings]
\label{rem:scaling}
We observe that the equation \eqref{equationL} is invariant with respect to the ``horizontal'' and ``vertical'' scale changes
\[
u_{\lambda,\mu}(t,x)=\mu\,u(\mu^{p-2}\,\lambda^p\,t,\lambda\,x),
\]
for every $\lambda,\mu>0$. Then it is easily seen that the a priori estimate \eqref{stimapriori} is invariant with respect to these scale changes. We point out that such estimate is the exact analogue of that for the evolutionary $p-$Laplacian, see \cite[Theorem 5.1, Chapter VIII]{Dib}.
Occasionally, in the paper we will work with {\it anisotropic} parabolic cubes of the type
\[
Q_R(t_0,x_0)=(t_0-R^p,t_0)\times (x_0-R,x_0+R)^N.
\]
The choice of cubes of this type could be loosely justified by a dimensional analysis of the equation. Indeed, by considering the quantity $u$ as dimensionless and using the family of scalings
\[
(t,x)\mapsto (\lambda^p\,t,\lambda\,x),\qquad \mbox{ for every } \lambda>0,
\]
we get the relation
\[
\mbox{ time } \sim (\mbox{length})^p.
\]
However, as it is well-known, estimates on cubes of the type $Q_R$ are too restrictive when looking at $C^{0,\alpha}$ estimates for $\nabla u$. Indeed, in light of the so-called {\it intrinsic geometry}, it is much more important to work with local estimates on cubes $\mathcal{Q}_{\tau,R}$, where the time scale $\tau$ is adapted to the solution itself: roughly  speaking, we can take
\[
\tau\sim R^2\,|\nabla u|^{p-2}.
\]
This explains the importance of having \eqref{stimapriori} with  two \emph{independent} scales $R$ and $\tau$.
We refer to \cite[Chapter VIII]{Dib} for a description of the method of intrinsic scalings, where these heuristics are clarified.
\end{rem}
\begin{rem}[Case $1<p\le 2$]
When \(p=2\), the orthotropic parabolic equation \eqref{equationL} boils down to the standard heat equation, for which solutions are well-known to be smooth. For this reason, in our statement we restrict our attention to the case $p>2$. However, we point out that by making the choice $\tau=R^2$ and taking the limit as $p$ goes to $2$ in \eqref{stimapriori}, we formally end up with the classical gradient estimate for solutions of the heat equation
\[
\|\nabla u\|_{L^\infty(\mathcal{Q}_{\sigma R^2,\sigma R}(t_0,x_0))}\le \frac{C}{(1-\sigma)^\frac{N+2}{2}}\,\left(\fint_{\mathcal{Q}_{R^2,R}(t_0,x_0)} |\nabla u|^2\,dt\,dx\right)^\frac{1}{2}.
\]
In light of the previous remark, in the case $p=2$ the relation 
\[
\mbox{ time } \sim (\mbox{length})^2,
\]
is now the natural one.
\par
As for the singular case $p<2$, this is somehow simpler than its degenerate counterpart. In this case, the local Lipschitz regularity of solutions to \eqref{equationL} can be directly inferred  from \cite[Theorem 1]{PV}, under the restriction 
\[
p>\frac{2\,N}{N+2}.
\] 
Indeed, the result of \cite{PV} covers (among others) the case of parabolic equations of the form
\[
u_t=\mathrm{div} \nabla F(\nabla u),
\]
under the following assumptions on the convex function $F$
\[
|\nabla F(z)|\le C\,|z|^{p-1}\qquad \mbox{ and }\qquad \langle D^2 F(z)\,\xi,\xi\rangle\ge \frac{1}{C}\,|z|^{p-2}\,|\xi|^2,\qquad \mbox{ for every }\xi\in\mathbb{R}^N, z\in \mathbb{R}^N\setminus \{0\}.
\]
It is not difficult to see that the orthotropic function \eqref{orto} for $p<2$ matches both requirements. Indeed, observe that 
\[
\langle D^2 F(z)\,\xi,\xi\rangle=(p-1)\,\sum_{i=1}^N |z_i|^{p-2}\,|\xi|^2\ge (p-1)\,|z|^{p-2}\,|\xi|^2,
\]
thanks to the fact that $p-2<0$.
Thus in the subquadratic case, the orthotropic structure helps more than it hurts, in a sense.
\end{rem}

\begin{rem}[Anisotropic diffusion]
We conclude this part by observing that, more generally, one could consider the following parabolic equation
\[
u_t=\sum_{i=1}^N \left(|u_{x_i}|^{p_i-2}\,u_{x_i}\right)_{x_i},\qquad \mbox{ in }\Omega\times I,
\]
which still has an orthotropic structure. Now we have a whole set of exponents $1<p_1\le p_2\le \dots p_N$, one for each coordinate direction. We cite the paper \cite{TT}, where some {\it global} Lipschitz regularity results are proven for solutions of the relevant Cauchy-Dirichlet problem, under appropriate regularity assumptions on the data. We point out that in light of their global nature, for $p_1=\dots=p_N=p>2$ such results are not comparable to ours.
We also refer to \cite{CMV} for a sophisticated Harnack inequality for positive local weak solutions, as well as for some further references on the problem. Finally, the very recent paper \cite{FVV} contains a thorough study of the Cauchy problem in the case $p_i<2$, together with some regularity results. 
\par
However, as for the counterpart of our Main Theorem for {\it local solutions} of this equation, this is still an open problem, to the best of our knowledge.
\end{rem}

\subsection{Technical aspects of the proof}
The core of the proof of the Main Theorem is an a priori Lipschitz estimate for smooth solutions of the orthotropic parabolic equation, see Proposition \ref{lm-apriori-estimate} below. More precisely, we introduce the regularized problem
\[
(u_\varepsilon)_t=\mbox{ div }\nabla F_\varepsilon (\nabla u_\varepsilon),
\]  
where \(F_\varepsilon\) is a smooth uniformly convex approximation of the orthotropic function \eqref{orto}.
By the classical regularity theory, the  maps \(u_\varepsilon\) are regular enough to justify all the calculations below. The goal is to establish a local uniform Lipschitz estimate on $u_\varepsilon$, which does not depend on the regularization parameter \(\varepsilon\). Finally, we let \(\varepsilon\) go to $0$ and prove that the family \(u_{\varepsilon}\) converges to the original solution \(u\). This allows to obtain the Lipschitz estimate for \(u\) itself.
\vskip.2cm\noindent
In the subsequent part of this section, we emphasize the main difficulties to get such a Lipschitz estimate on \(u_{\varepsilon}\). In order to simplify the presentation, we drop the index \(\varepsilon\) both for \(F_\varepsilon\) and \(u_\varepsilon\). The strategy is apparently quite classical: we rely on a  Moser iterative scheme of reverse H\"older's inequalities, resulting from  the interplay between Cacciopoli estimates and the Sobolev embeddings.
\par
To be more specific, we first differentiate the equation  with respect to a spatial variable \(x_j\), so as to get the equation solved by the $j-$th component of the gradient. This is given by
\begin{equation}\label{eq169}
\iint_{I\times \Omega}(u_{x_j})_t \,\varphi\,dt\,dx + \int_{I\times \Omega}\langle (\nabla F(\nabla u))_{x_j}, \nabla \varphi\rangle \,dt\,dx=0, \qquad \mbox{ for every }\varphi\in C^{\infty}_0(I\times \Omega).
\end{equation}
More generally,  the composition of the component $u_{x_j}$ with a non-negative convex function $h$ is a subsolution of this equation. Accordingly, the map $h(u_{x_j})$ satisfies 
the Caccioppoli inequality which is naturally attached to \eqref{eq169} (see Lemma \ref{lm:standard} below). If $I=(T_0,T_1)$ and $\tau\in (T_0,T_1)$, this reads as follows:
\begin{equation}
\begin{split}
\chi(\tau)\,\int_{\{\tau\}\times\Omega} h^{2}(u_{x_j})\,\eta^2\,dx&+
\iint_{(T_0,\tau)\times \Omega}\langle D^{2}F(\nabla u)\nabla h(u_{x_j}), \nabla h(u_{x_j}) \rangle\, \chi\,\eta^2\,dt\,dx \\
&\lesssim \int_{(T_0,\tau)\times \Omega} \chi'\, \eta^2\, h^{2}(u_{x_j})\,dt \,dx \\
&+  \iint_{(T_0,\tau)\times \Omega}\langle D^{2}F(\nabla u)\nabla \eta, \nabla \eta \rangle \, h^{2}(u_{x_j})\,\chi\,dt\,dx.
\label{eq177}
\end{split}
\end{equation}
Here, the maps \(\chi\in C^{\infty}_0(I)\) and \(\eta\in C^{\infty}_0(\Omega)\) are non-negative cut-off functions in the time and space variables, respectively. We have used the following expedient notation: given  $f\in L^1(I\times\Omega)$, 
\[
\int_{\{\tau\}\times\Omega} f\,dx:=\int_\Omega f(\tau,x)\,dx,\qquad \mbox{ for a.\,e. } \tau\in I.
\]
When \(F\) is a uniformly elliptic integrand, in the sense that
\[
\frac{1}{C}\,|\xi|^2\le \langle D^2 F(\nabla u)\,\xi,\xi\rangle\le C\,|\xi|^2,\qquad \mbox{ for every } \xi\in \mathbb{R}^N,
\] 
one can easily obtain from \eqref{eq177} a crucial ``unnatural'' feature of the subsolution $h(u_{x_j})$; that is, a sort of reverse Poincar\'e inequality where the Sobolev norm of $h(u_{x_j})$ is controlled by the $L^2$ norm of the subsolution itself. In conjunction with the Sobolev inequality, this is the cornerstone which eventually leads to
 the classical version of the Moser iterative scheme. It should be noticed that this strategy still  works even in the degenerate case, provided the Hessian behaves like
\[
\frac{1}{C}\,|\nabla u|^{p-2}\,|\xi|^2\le \langle D^2 F(\nabla u)\,\xi,\xi\rangle\le C\,|\nabla u|^{p-2}\,|\xi|^2,\qquad \mbox{ for every } \xi\in\mathbb{R}^N,
\]
as for the evolutionary $p-$Laplace equation \eqref{equationP}.
It is sufficient to use the ``absorption of degeneracy'' trick, where the degenerate weight $|\nabla u|^{p-2}$ is recombined with the subsolution $h(u_{x_j})$ by means of simple algebraic manipulations. This  still permits to infer from \eqref{eq177} a control on the Sobolev norm of a suitable convex function of $u_{x_j}$.
This is nowadays a standard technique in the field; for the elliptic case, it goes back to the pioneering works by Ural'tseva \cite{Ur} and Uhlenbeck \cite{Uh}.    
\vskip.2cm\noindent
As we explained above, due to the severe degeneracy of \(D^{2}F\) in our orthotropic situation, {\it it is not possible} to follow the same path. In order to rely on such an absorption trick, we have to go through a {\it tour de force}
and to introduce a new family of \emph{weird} Caccioppoli inequalities (see Lemma \ref{lm:weird} below). These are the parabolic counterparts of a corresponding estimate introduced in the elliptic setting in \cite{BouBra} and then fruitfully exploited in \cite{BouBraLeoVer}. 
\par
The crucial idea is to mix  together the components of the gradient with respect to 2 orthogonal directions. This  compensates the lack of ellipticity of \(D^2F\) and allows to rely on the Sobolev embeddings in the iterative scheme. 
We do not detail these Caccioppoli-type estimates here, but instead explain the main additional difficulties with respect to  the elliptic framework. 

Let us come back for one instant to the standard Caccioppoli inequality \eqref{eq177}. It follows from \eqref{eq169} by taking  \(\varphi= h\,h'(u_{x_j})\,\chi\, \eta^2\). In particular, the parabolic term is given by 
\[
\iint_{I\times\Omega}(u_{x_j})_t \,\varphi\,dt\,dx = \frac{1}{2}\iint_{I\times\Omega}(h^{2}(u_{x_j}))_t \, \chi \,\eta^2\,dt\,dx. 
\]
Then an integration by parts yields
\[
\iint_{I\times\Omega}(u_{x_j})_t \,\varphi\,dt\,dx = -\frac{1}{2}\iint_{I\times\Omega}h^{2}(u_{x_j})\,\chi'\,\eta^2\,dt\,dx. 
\] 
The latter yields  the ``time slice'' term on the left-hand side of \eqref{eq177}.
This way, the time derivative is transferred to \(\chi\) and one can handle the factor \(h^{2}(u_{x_j})\) as in  the elliptic framework.

For the weird Cacciopoli inequalities, the test function is now \(\varphi = u_{x_j}\,\Phi'(u_{x_j}^2)\,\Psi(u_{x_k}^2)\,\chi\, \eta^2\), for some \(1\leq j, k \leq N\). The corresponding parabolic term becomes:
\[
\iint_{I\times\Omega}(u_{x_j})_t \,\varphi\,dt\,dx = \frac{1}{2}\iint_{I\times\Omega}\big(\Phi(u_{x_j}^2)\big)_t\,\Psi(u_{x_k}^2) \, \chi \,\eta^2\,dt\,dx. 
\]
In contrast to the previous situation, we cannot perform an integration by parts  to get rid of the time derivative on \(\Phi(u_{x_j}^2)\), since it would affect the factor \(\Psi(u_{x_k}^2)\). In order to overcome this difficulty, which does not arise in the elliptic setting, we need a new approach, aimed at ``symmetrizing'' the above quantity containing $u_{x_j}$ and $u_{x_k}$.

Basically, we merge together two weird Cacciopoli inequalities, where the spatial variables \(x_j\) and \(x_k\) play symmetric roles. More specifically, we insert into \eqref{eq169} the test functions:
\[
\varphi = u_{x_j}\,\Phi'(u_{x_j}^2)\,\Psi(u_{x_k}^2)\,\chi\, \eta^2, \qquad \widetilde{\varphi} = u_{x_k}\,\Psi'(u_{x_k}^2)\,\Phi(u_{x_j}^2)\,\chi\, \eta^2,
\]      
and then add the two resulting inequalities. The parabolic term is now replaced by the following quantity:
\[
\frac{1}{2}\,\iint_{I\times\Omega}\left(\Phi(u_{x_j}^2)\,\Psi(u_{x_k}^2) \right)_t\,\chi\,\eta^2\,dt\,dx. 
\]
This allows to integrate by parts and transfer the time derivative on the test function. It turns out that by a suitable adaptation of  the arguments that we used in the elliptic case, one can incorporate this new  term in the iterative Moser scheme. This finally leads to the desired local $L^\infty$ estimate on $\nabla u$.

\subsection{Plan of the paper}

The paper is organized as follows: after collecting the basic terminology and some preliminaries  on Steklov averages in Section \ref{sec:2}, we present in Section \ref{sec:3} the proofs of the new Cacciopoli inequalities in the parabolic setting. We detail the iterative Moser scheme  in Section \ref{sec:4} and finally establish the Main Theorem in Section \ref{sec:5}, by transferring to the original solution \(u\) the a priori estimates obtained on the approximating solutions \(u_{\varepsilon}\).

\begin{ack}
We thank Alkis S. Tersenov for pointing out the reference \cite{TT}.
The paper has been partially written during a visit of P.\,B. \& L.\,B. to Napoli and of C.\,L. to Ferrara. Both visits have been funded by the Gruppo Nazionale per l'Analisi Matematica, la Pro\-ba\-bi\-li\-t\`a
e le loro Applicazioni (GNAMPA) through the project ``{\it Regolarit\`a per operatori degeneri con crescite generali\,}''. 
Hosting institutions are gratefully acknowledged. 
\par
C.\,L. and A.\,V. are members of the Gruppo Nazionale per l'Analisi Matematica, la Probabilit\`a
e le loro Applicazioni (GNAMPA) of the Istituto Nazionale di Alta Matematica (INdAM).
\end{ack}

\section{Preliminaries}
\label{sec:2}

\subsection{Local solutions}\label{subsection:local-solutions}
Let $\Omega\subset\mathbb{R}^N$ be an open bounded set and $I\subset \mathbb{R}$ an open bounded interval. 
Fix $p>2$ and take $\mathcal{A}:\mathbb{R}^N\to\mathbb{R}^N$ a continuous function such that
\[
\langle \mathcal{A}(z)-\mathcal{A}(w),z-w\rangle \ge 0,\qquad \mbox{ for every } z,w\in\mathbb{R}^N,
\]
and 
\[
\langle \mathcal{A}(z),z\rangle\ge \frac{1}{C}\,|z|^p,\qquad \mbox{ and }\qquad |\mathcal{A}(z)|\le C\,|z|^{p-1},\qquad \mbox{ for every }z\in\mathbb{R}^N.
\]
 We say that $u\in L^p_{\rm loc}(I;W^{1,p}_{\rm loc}(\Omega))$ is a {\it local weak solution} of the quasilinear diffusion equation
\begin{equation}
\label{equazione}
u_t=\mathrm{div\,}\mathcal{A}(\nabla u),\qquad \mbox{ in }I\times\Omega,
\end{equation}
if for every $\varphi\in C^\infty_0(I\times \Omega)$, we have
\[
-\iint_{I\times\Omega} u\,\varphi_t\,dt\,dx+\iint_{I\times \Omega} \langle \mathcal{A}(\nabla u),\nabla\varphi\rangle\,dt\,dx=0.
\]

\subsection{Steklov averages}

Throughout the paper, we denote by  \(T_0<T_1\) the endpoints of the time interval \(I\).
Let $v\in L^1_{\rm loc}(I\times\Omega)$. 
For every \(0<\sigma< T_1-T_0\), we define its so-called 
\begin{itemize}
\item {\it forward Steklov average}
\[
v_\sigma^+(t,x)=\fint_t^{t+\sigma} v(\tau,x)\,d\tau, \qquad \mbox{ for every }(t,x)\in (T_0, T_1-\sigma)\times\Omega;
\]
\item {\it backward Steklov average}
\[
v_\sigma^-(t,x)=\fint_{t-\sigma}^t v(\tau,x)\,d\tau, \qquad \mbox{ for every } (t,x)\in (T_0+\sigma, T_1)\times\Omega.
\]
\end{itemize}
We shall use some standard properties of the Steklov averages. Let \(0<\sigma<T_1-T_0\) and   \(\psi\in L^{\infty}(I\times\Omega)\) such  that \(\psi\) is compactly supported in \((T_0, T_1-\sigma)\times\Omega\). We extend \(\psi\) by \(0\) on \((\mathbb{R}\setminus I)\times \Omega\),  
so that \(\psi_{\sigma}^-\) is well-defined on \(I\times \Omega\) and compactly supported therein. 
By the Fubini theorem, we have
\begin{equation}\label{eq-Milano}
\iint_{(T_0, T_1-\sigma)\times\Omega} v_{\sigma}^+ \,\psi \,dt\,dx = \iint_{I\times\Omega} v\,\psi_{\sigma}^- \,dt\,dx.
\end{equation}
Moreover, if \(v\in L^{q}_{\rm loc }(I\times\Omega)\) for some \(1\le q<\infty\), then \(v_{\sigma}^+\) converges to \(v\) in \(L^{q}_{\rm loc}(I\times\Omega)\), as \(\sigma\) goes to \(0\), see e.g. \cite[Chapter I, Lemma 3.2]{Dib}.
\par
Finally, we can derive from \eqref{eq-Milano} the following regularity properties of the Steklov averages:

\begin{lm}\label{lm-Siena}
Let $v\in L^1_{\rm loc}(I\times\Omega)$. Then for every \(0<\sigma<T_1-T_0\),
\begin{enumerate}
\item the map \(v_{\sigma}^+\) belongs to \(W^{1,1}_{\rm loc}((T_0, T_1-\sigma); L^{1}_{\rm loc}(\Omega))\) and 
\begin{equation}
\label{eq-Pavia}
(v_\sigma^+)_t(t,x)=\frac{v(t+\sigma,x)-v(t,x)}{\sigma}, \qquad \mbox{ for a.e. } (t,x)\in (T_0, T_1-\sigma)\times\Omega;
\end{equation} 
\item if one further assumes that \(v\in L^{1}_{\rm loc}(I; W^{1,1}_{\rm loc}(\Omega))\), then
\(\nabla (v_\sigma^+) \in L^{1}_{\rm loc}((T_0, T_1-\sigma)\times\Omega)\) and
\begin{equation}
\label{eq-Brescia}
\nabla (v_\sigma^+) = (\nabla v)_{\sigma}^+.
\end{equation}
\end{enumerate}
\end{lm}
\begin{proof}
Fix \(0<\sigma<T_1-T_0\).
Let \(\psi \in C^{\infty}_0((T_0, T_1-\sigma)\times\Omega)\). Then by \eqref{eq-Milano},
\[
\iint_{(T_0, T_1-\sigma)\times\Omega}v_\sigma^+ \,\psi_t\,dt\,dx 
=\iint_{I\times\Omega}v\,(\psi_t)_{\sigma}^-\,dt\,dx
= \iint_{I\times\Omega}v(t,x)\,\frac{\psi(t,x)-\psi(t-\sigma,x)}{\sigma}\,dt\,dx.
\]
By an obvious change of variables, this yields
\[
\iint_{(T_0, T_1-\sigma)\times\Omega}v_\sigma^+ \,\psi_t\,dt\,dx = -\iint_{(T_0, T_1-\sigma)\times\Omega}\frac{v(t+\sigma,x)-v(t,x)}{\sigma}\,\psi(t,x)\,dt\,dx,
\]
which gives the desired identity \eqref{eq-Pavia}.
\par
In order to prove \eqref{eq-Brescia}, we rely again on \eqref{eq-Milano}, this time tested with \(\psi_{x_j}\) in place of \(\psi\), for some \(1\leq j \leq N\):
\[
\iint_{(T_0, T_1-\sigma)\times\Omega}v_\sigma^+ \,\psi_{x_j}\,dt\,dx = \iint_{I\times\Omega}v\, (\psi_{x_j})_{\sigma}^-\, dt\,dx=\iint_{I\times\Omega}v\, (\psi_{\sigma}^-)_{x_j}\, \,dt\,dx.
\]
In the last equality, we have derived under the integral sign  the smooth function \(\psi\). Hence, by integrating by parts the last integral and using \eqref{eq-Milano} again, one gets 
\[
\iint_{I\times\Omega}v_\sigma^+ \,\psi_{x_j}\,dt\,dx =  -\iint_{(T_0, T_1-\sigma)\times \Omega}(v_{x_j})_{\sigma}^+\psi \,dt\,dx,
\]
from which \eqref{eq-Brescia} follows.
\end{proof}

\section{Energy estimates for a regularized equation}
\label{sec:3}

\subsection{An approximating equation}
We denote by 
\[
G(\xi)=\frac{1}{p}\left(1+|\xi|^2\right)^\frac{p}{2},\qquad \mbox{ for every } \xi\in\mathbb{R}^N,
\]
and for every \(\varepsilon \in (0,1)\), we consider the convex function 
\begin{equation}
\label{Fe}
F_{\varepsilon}(\xi) = \frac{1}{p}\,\sum_{i=1}^{N}|\xi_i|^{p} + \varepsilon\ G(\xi),\qquad \mbox{ for every } \xi\in\mathbb{R}^N.
\end{equation}
We consider a local weak solution \(u_{\varepsilon}\in L^{p}_{\rm loc}(I;W^{1,p}_{\rm loc}(\Omega))\) of the equation \eqref{equazione} with the choice
\[
\mathcal{A}(z)=\nabla F_\varepsilon(z).
\]
This means that $u_\varepsilon$ verifies
\begin{equation}
\label{weak-again}
-\iint_{I\times\Omega} u_{\varepsilon}\,\varphi_t\,dt\,dx+\iint_{I\times\Omega} \langle \nabla F_{\varepsilon}(\nabla u_{\varepsilon}) , \nabla \varphi\rangle \,dt\,dx=0,
\end{equation}
for every $\varphi\in C^\infty_0(I\times\Omega)$. Observe that the map \(F_\varepsilon\) belongs to \(C^{2}(\mathbb{R}^N)\) and satisfies
\[
\varepsilon\, (1+|\xi|^2)^\frac{p-2}{2}|\zeta|^2\leq \langle D^2 F_\varepsilon(\xi)\, \zeta, \zeta\rangle \leq (1+\varepsilon)\,(p-1)\,(1+|\xi|^2)^\frac{p-2}{2}|\zeta|^2, \qquad \mbox{ for every }\, \xi, \zeta \in \mathbb{R}^N.
\]
Hence, one can rely on  the classical regularity theory for quasilinear parabolic equations, see e.g. \cite[Theorem 5.1, Chapter VIII]{Dib} and \cite[Lemma 3.1]{BDM}, to get: 
\begin{equation}
\label{eq327}
\nabla u_\varepsilon\in L^\infty_{\rm loc} (I\times\Omega ) \qquad \mbox{ and } \qquad u_\varepsilon\in L^{2}_{\rm loc}(I; W^{2,2}_{\rm loc}(\Omega)).
\end{equation}
In the following computations, we delete the index \(\varepsilon\) both for \(u\) and \(F\).

\subsection{An equation for the spatial gradient}

In order to establish a Lipschitz bound on our solution \(u\), we need to differentiate \eqref{weak-again} with respect to the spatial variables \(x_j\), \(1\leq j \leq N\). 
\vskip.2cm
Fix \(0<\sigma<T_1-T_0\). Let $\psi\in C^\infty_0((T_0,T_1-\sigma)\times\Omega)$. As already observed, the backward Steklov average
\[
\varphi(t,x)=\psi_\sigma^-(t,x),\qquad \mbox{ for } (t,x)\in I\times\Omega,
\]
is compactly supported in $I\times\Omega$. 
We can thus insert it in \eqref{weak-again}: 
\[
-\iint_{I\times\Omega} u\, (\psi_\sigma^-)_t\,dt\,dx 
+\iint_{I\times\Omega}\langle \nabla F(\nabla u) , \nabla \psi_{\sigma}^-\rangle\,dt\,dx=0.
\]
Since \((\psi_\sigma^-)_t=(\psi_t)_{\sigma}^-\),  \eqref{eq-Milano} implies that
\[
-\iint_{I\times\Omega} u\,(\psi_\sigma^-)_t\,dt\,dx =-\iint_{(T_0, T_1-\sigma)\times\Omega} u_{\sigma}^+ \, \psi_t \,dt\,dx  = \iint_{(T_0, T_1-\sigma)\times\Omega} (u_{\sigma}^+)_t \, \psi \,dt\,dx.
\]
One thus gets
\begin{equation}
\label{coltempo}
\iint_{(T_0, T_1-\sigma)\times\Omega}(u_{\sigma}^+)_t\, \psi \,dt\,dx
+\iint_{I\times\Omega}\langle \nabla F(\nabla u) , \nabla \psi_{\sigma}^-\rangle\,dt\,dx=0,
\end{equation}
for every $\psi\in C^\infty_0((T_0,T_1-\sigma)\times\Omega)$. 
\vskip.2cm\noindent
Let \(j \in \{1, \dots, N\}\) and  \(\varphi\in C^{\infty}_0((T_0,T_1-\sigma)\times\Omega)\).  
The map
\[
(u_{\sigma}^+)_t (t,x)= \frac{u(t+\sigma,x)-u(t,x)}{\sigma},
\]
belongs to \(L^{p}_{\rm loc}((T_0, T_1-\sigma);W^{1,p}_{\rm loc}(\Omega))\) and \(((u_{\sigma}^+)_t)_{x_j}=((u_{x_j})_{\sigma}^+)_{t}\). By derivation under the integral sign, one also has 
\[
\nabla ( (\varphi_{x_j})_{\sigma}^-) = (\nabla (\varphi_{\sigma}^-))_{x_j}.
\]
We insert \(\psi = \varphi_{x_j}\) in equation \eqref{coltempo}. 
An integration by parts in the spatial variable leads to
\[
\iint_{(T_0, T_1-\sigma)\times\Omega}\left((u_{x_j})_{\sigma}^+\right)_{t} \varphi \,dt\,dx
+\iint_{I\times\Omega}\langle (\nabla F(\nabla u))_{x_j} , \nabla \varphi_{\sigma}^-\rangle\,dt\,dx=0.
\]
Finally, using \eqref{eq-Milano} in the second term, one gets
\begin{equation}
\label{eq1043}
\iint_{(T_0, T_1-\sigma)\times\Omega}\left((u_{x_j})_{\sigma}^+\right)_{t} \varphi \,dt\,dx
+\iint_{(T_0, T_1-\sigma)\times\Omega}\langle \left((\nabla F(\nabla u))_{x_j}\right)_\sigma^+ , \nabla \varphi\rangle\,dt\,dx=0.
\end{equation}
We observe that, since \(F\in C^{2}(\mathbb{R}^N)\) and \(\nabla u\in L^{\infty}_{\rm loc}(I\times\Omega)\cap L^{2}_{\rm loc}(I; W^{1,2}_{\rm loc}(\Omega))\), one has 
\[
\nabla F(\nabla u)\in L^{2}_{\rm loc}((0,T);W^{1,2}_{\rm loc}(\Omega)).
\] 
We can thus appeal to a density argument to get that \eqref{eq1043} remains true for every \(\varphi\in L^{2}(I;W^{1,2}(\Omega))\), with compact support in \((T_0,T_1-\sigma)\times \Omega\).

\subsection{Caccioppoli-type inequalities}

As explained in the introduction, the first technical tool in the proof of the Lipschitz bound of \(u\) is the following Cacciopoli inequality which provides a \(W^{1,2}\) estimate on \(h(u_{x_j})\), where \(h\) is any smooth convex function. 
\begin{lm}[Standard Caccioppoli inequality]
\label{lm:standard}
Let  $\eta\in C^{\infty}_0(\Omega)$ and $\chi \in C^{\infty}_0((T_0,T_1])$ be two non-negative functions, with $\chi$ non-decreasing. Let $h:\mathbb{R}\to \mathbb{R}$ be a $C^{1}$ convex non-negative function. Then, for almost every $\tau\in I$ and every $j=1,...,N$, we have
\begin{equation}
\begin{split}
\chi(\tau)\int_{\{\tau\}\times\Omega}h^2(u_{x_j})\,  \eta^2\,dx&
+\iint_{(T_0,\tau)\times\Omega} \langle D^2F(\nabla u)\,\nabla h(u_{x_j}), \nabla h(u_{x_j})\rangle\, \chi\, \eta^2\,dt\,dx \\
&\leq \iint_{(T_0,\tau)\times\Omega} \chi'\, \eta^2\, h^2(u_{x_j})\,dt\,dx\\
&+4\,\iint_{(T_0,\tau)\times\Omega}\langle D^2F(\nabla u)\,\nabla \eta, \nabla \eta\rangle\, h^2(u_{x_j})\,\chi \,dt\,dx.\label{eq1218}
\end{split}
\end{equation}
\end{lm}
\begin{proof}We first assume that  $h$ is a $C^{2}$ convex non-negative function. 
Let $\zeta \in C^{\infty}_0(I)$ and \(\eta\in C^{\infty}_0(\Omega)\). There exists \(0<\sigma_1<(T_1-T_0)/2\) such that \(\zeta\) is compactly supported in \((T_0+\sigma_1, T_1-\sigma_1)\).
Given \(0<\sigma<\sigma_1\), Lemma \ref{lm-Siena} and \eqref{eq327} imply that \((u_{x_j})_{\sigma}^+ \in W^{1,1}_{\rm loc}((T_0, T_1-\sigma); L^{1}_{\rm loc}(\Omega))\cap L^{\infty}_{\rm loc}((T_0, T_1-\sigma)\times\Omega)\). Hence, the map \(h^{2}((u_{x_j})_{\sigma}^+ )\) belongs to \(W^{1,1}_{\rm loc}((T_0, T_1-\sigma); L^{1}_{\rm loc}(\Omega) )\) and we have
\begin{equation}\label{eq-Genova}
\frac{1}{2} \left(h^2\left(\left(u_{x_j}\right)_{\sigma}^+\right) \right)_t = (h\,h')\left(\left(u_{x_j}\right)_{\sigma}^+\right)\,\left((u_{x_j})_{\sigma}^+\right)_{t}.
\end{equation}
We insert in \eqref{eq1043} the test function
\[
\varphi = (h\,h')\left(\left(u_{x_j}\right)_{\sigma}^+\right)\,\zeta\, \eta^2,
\]
which has compact support in \((T_0+\sigma_1,T_1-\sigma_1)\times\Omega\) and belongs to \(L^{\infty}((T_0+\sigma_1,T_1-\sigma_1)\times\Omega)\cap L^{2}((T_0+\sigma_1,T_1-\sigma_1);W^{1,2}(\Omega))\).
By \eqref{eq-Genova},
\[
\left((u_{x_j})_{\sigma}^+\right)_{t} \varphi  
=
\frac{1}{2} \left(h^2\left(\left(u_{x_j}\right)_{\sigma}^+\right) \right)_t\,\zeta\, \eta^2.
\]
We use the above identity to infer:
\[
\frac{1}{2}\iint_{(T_0+\sigma_1,T_1-\sigma_1)\times\Omega}\left( h^2\left(\left(u_{x_j}\right)_{\sigma}^+\right) \right)_t\,\zeta\, \eta^2\,dt\,dx 
+\iint_{(T_0+\sigma_1,T_1-\sigma_1)\times\Omega}\langle \left((\nabla F(\nabla u))_{x_j}\right)_\sigma^+ , \nabla \varphi\rangle\,dt\,dx
=0.
\]
We then perform an integration by parts with respect to the time variable in the first term:
\[
-\frac{1}{2}\iint_{(T_0+\sigma_1,T_1-\sigma_1)\times\Omega}h^{2}\left(\left(u_{x_j}\right)_{\sigma}^+\right) \zeta' \eta^2\,dt\,dx 
+\iint_{(T_0+\sigma_1,T_1-\sigma_1)\times\Omega}\langle \left((\nabla F(\nabla u))_{x_j}\right)_{\sigma}^+ , \nabla \varphi\rangle\,dt\,dx=0.
\]
We now want to take the limit as $\sigma$ goes to $0$.  
Let \(\Omega_1\Subset \Omega\) such that \(\eta\) is compactly supported in \(\Omega_1\).
Since \(u_{x_j}\in L^{2}_{\rm loc}(I\times\Omega)\), we have
\[
\lim_{\sigma\to 0^+}\|\left(u_{x_j}\right)_{\sigma}^+-u_{x_j}\|_{L^2((T_0+\sigma_1,T_1-\sigma_1)\times\Omega_1)}=0.
\] 
Moreover, we know that \(u_{x_j}\in L^{\infty}_{\rm loc}(I\times\Omega)\) which guarantees that
there exists \(C_1>0\) such that for every \(\sigma\in (0, \sigma_1)\), 
\[
\left|\left(u_{x_j}\right)_{\sigma}^+\right|\leq C_1,\qquad \mbox{ a.\,e. on } (T_0+\sigma_1,T_1-\sigma_1)\times\Omega_1.
\] 
It then follows from the Dominated Convergence Theorem that 
\begin{equation}\label{eq443}
\lim_{\sigma\to 0^+}-\frac{1}{2}\iint_{(T_0+\sigma_1,T_1-\sigma_1)\times\Omega}h^{2}\left(\left(u_{x_j}\right)_{\sigma}^+\right) \zeta' \eta^2\,dt\,dx =-\frac{1}{2}\iint_{I\times\Omega}h^{2}\left(u_{x_j}\right) \zeta' \eta^2\,dt\,dx. 
\end{equation}
Next, by recalling the choice of $\varphi$ above, we have
\[
\nabla \varphi= (h\,h')'\left((u_{x_j})_{\sigma}^+\right)\,\nabla\left( (u_{x_j})_{\sigma}^+ \right)\,\zeta\, \eta^2 + (h\,h')\left((u_{x_j})_{\sigma}^+\right)\,\zeta\,\nabla(\eta^2).
\]
By Lemma \ref{eq-Brescia}, we know that \(\nabla\left( (u_{x_j})_{\sigma}^+ \right)=( \nabla u_{x_j})_{\sigma}^+\).
This implies that   \(\nabla\left( (u_{x_j})_{\sigma}^+ \right)\) converges to \(\nabla u_{x_j}\) in \(L^{2}((T_0+\sigma_1,T_1-\sigma_1)\times\Omega)\). Hence, a similar argument to the one leading to \eqref{eq443} implies that 
\[
\lim_{\sigma\to 0^+} \Big\|\nabla \varphi-\nabla \left((h\,h')(u_{x_j})\, \zeta\, \eta^2\right)\Big\|_{L^2((T_0+\sigma_1,T_1-\sigma_1)\times\Omega)}=0.
\]
Finally,  by using that \((\nabla F(\nabla u))_{x_j}\in L^{2}_{\rm loc}(I\times\Omega)\), we can infer that
\[
\lim_{\sigma\to 0^+}\Big\|\left((\nabla F(\nabla u))_{x_j}\right)_{\sigma}^{+}-(\nabla F(\nabla u))_{x_j}\Big\|_{L^{2}((T_0+\sigma_1,T_1-\sigma_1)\times\Omega_1)}=0.
\]
It follows that 
$$
\lim_{\sigma\to 0^+}\iint_{(T_0+\sigma_1,T_1-\sigma_1)\times\Omega}\langle \left((\nabla F(\nabla u))_{x_j}\right)_{\sigma}^+ , \nabla \varphi \rangle\,dt\,dx=\iint_{I\times\Omega}\langle (\nabla F(\nabla u))_{x_j} , \nabla \left((h\,h')(u_{x_j})\, \zeta\, \eta^2\right)\rangle\,dt\,dx.
$$
Up to now, we have thus proved: 
\begin{equation}\label{eq110511}
-\frac{1}{2}\iint_{I\times\Omega}h^{2}\left(u_{x_j}\right)\, \zeta'\, \eta^2\,dt\,dx 
+\iint_{I\times\Omega}\langle (\nabla F(\nabla u))_{x_j} , \nabla \left((h\,h')(u_{x_j})\, \zeta\, \eta^2\right)\rangle\,dt\,dx
=0. 
\end{equation}
We now choose \(\zeta\) as follows. Let \(\chi \in C^{\infty}_0((T_0,T_1])\) be as in the statement.
Given \(\tau\in I\) and \(\delta>0\) such that \(T_0<\tau <\tau+\delta<T_1\), we define
\[
\widetilde{\chi}_\delta(t):=
\left\{\begin{array}{ll}
1,& \textrm{ if } t\leq \tau,\\
1-\dfrac{t-\tau}{\delta},& \textrm{ if } \tau < t <\tau +\delta,\\
0,& \textrm{ if } t\geq \tau+\delta.
\end{array}
\right.
\]
We then insert 
\begin{equation}\label{eq_zeta}
\zeta(t)=\widetilde{\chi}_\delta(t)\,\chi(t),
\end{equation} 
in \eqref{eq110511}. Then, for almost every \(\tau \in I\), we can let \(\delta\) go to $0$ and obtain 
\begin{equation}
\label{eq1097}
\begin{split}
\frac{\chi(\tau)}{2}\,\int_{\{\tau\}\times\Omega}h^2(u_{x_j})\,  \eta^2\,dx
&+\iint_{(T_0,\tau)\times \Omega}\langle (\nabla F(\nabla u))_{x_j} , \nabla \left((h\,h')(u_{x_j})\, \chi\, \eta^2\right)\rangle\,dt\,dx\\
&= \frac{1}{2}\iint_{(T_0,\tau)\times \Omega} \chi'\, \eta^2\, h^2(u_{x_j})\,dt\,dx.
\end{split}
\end{equation}
Since $\chi$ does not depend on the spatial variable, we have
\[
\begin{split}
\langle (\nabla F(\nabla u))_{x_j} , \nabla \left((hh')(u_{x_j})\, \chi\, \eta^2\right)\rangle
&=\langle D^2F(\nabla u)\,\nabla h(u_{x_j}), \nabla h(u_{x_j})\rangle\, \chi\, \eta^2 \\
&+\langle D^2F(\nabla u)\,\nabla u_{x_j}, \nabla u_{x_j}\rangle\, h''(u_{x_j})\,h(u_{x_j})\,\chi\, \eta^2\\
&+2\,\langle D^2F(\nabla u)\,\nabla u_{x_j}, \nabla \eta\rangle\, (h\,h')(u_{x_j})\,\chi\, \eta.
\end{split}
\]
Since the second term is non-negative, by dropping it, we get from \eqref{eq1097}
\[
\begin{split}
\frac{\chi(\tau)}{2}\,\int_{\{\tau\}\times\Omega}h^2(u_{x_j})\,  \eta^2\,dx
&+\iint_{(T_0,\tau)\times \Omega}\langle D^2F(\nabla u)\,\nabla h(u_{x_j}), \nabla h(u_{x_j})\rangle\, \chi\, \eta^2\,dt\,dx\\
&\leq \frac{1}{2}\iint_{(T_0,\tau)\times \Omega} \chi' \,\eta^2\, h^2(u_{x_j})\,dt\,dx\\
&-2\iint_{(T_0,\tau)\times \Omega}\langle D^2F(\nabla u)\,\nabla u_{x_j}, \nabla \eta\rangle\, (h\,h')(u_{x_j})\,\chi\, \eta\,dt\,dx.
\end{split}
\]
In order to estimate the last term, we use the Cauchy-Schwarz inequality:
\[
\Big|\langle D^2F(\nabla u)\,\nabla u_{x_j} , \nabla \eta\rangle\Big| \leq \bigg(\langle D^2F(\nabla u)\,\nabla u_{x_j} , \nabla u_{x_j}\rangle  \bigg)^{\frac{1}{2}} \bigg( \langle D^2F(\nabla u)\,\nabla \eta , \nabla \eta\rangle\bigg)^{\frac{1}{2}}.
\]
A further application of Young inequality leads to
\[
\begin{split}
\Bigg|-2 \iint_{(T_0,\tau)\times \Omega}& \langle D^2F(\nabla u)\,\nabla u_{x_j} , \nabla \eta\rangle\,(h\,h')(u_{x_j})\,\chi\,\eta \,dt\,dx\Bigg|\\
&\leq \frac{1}{2}\, \iint_{(T_0,\tau)\times \Omega}\langle D^2F(\nabla u)\nabla u_{x_j},\nabla u_{x_j}\rangle\, h'(u_{x_j})^2\, \chi\,\eta^2\,dt\,dx\\
&+ 2\,  \iint_{(T_0,\tau)\times \Omega}\langle D^2F(\nabla u)\nabla \eta , \nabla \eta\rangle\, h^2(u_{x_j})\,\chi\,dt\,dx.
\end{split}
\]
In this way, the integral containing $\nabla u_{x_j}$ can be absorbed in the left hand side. 
Let us finally observe that we can remove the $C^2$ assumption on the function $h$, by a standard approximation argument.
\end{proof}

We next establish the key tool for the proof of our main result, namely  a Cacciopoli type inequality,  where two different partial derivatives \(u_{x_j}\) and \(u_{x_k}\) come into play.
 
\begin{lm}[Weird Caccioppoli inequality]
\label{lm:weird}
Let  $\eta\in C^{\infty}_0(\Omega)$ and $\chi \in C^{\infty}_0((T_0,T_1])$ be two non-negative functions, with $\chi$ non-decreasing. Let \(\Phi : \mathbb{R}^+\to \mathbb{R}\) and \(\Psi:\mathbb{R}^+\to \mathbb{R}\) be two \(C^1\) non-decreasing and non-negative convex functions. Then, for almost every $\tau\in I$, every $k, j=1,...,N$ and every $\theta\in[0,1]$, we have
\[
\begin{split}
\chi(\tau)&\int_{\{\tau\}\times\Omega}\Phi(u_{x_j}^2)\, \Psi(u_{x_k}^2)\, \eta^2\,dx
+\iint_{(T_0,\tau)\times \Omega} \langle D^2F(\nabla u)\nabla u_{x_j},\nabla u_{x_j}\rangle\,\Phi'(u_{x_j}^2)\,  \Psi(u_{x_k}^2)\,\chi\,\eta^2 \,dt\,dx \\
&\leq \iint_{(T_0,\tau)\times \Omega} \chi'\, \eta^2\, \Phi(u_{x_j}^2)\,\Psi(u_{x_k}^2)\,dt\,dx\\
&+4\iint_{(T_0,\tau)\times \Omega} \langle D^2F(\nabla u)\nabla \eta , \nabla \eta\rangle\,\left(u_{x_j}^2\,\Phi'(u_{x_j}^2) \,\Psi(u_{x_k}^2)+u_{x_k}^2\,\Psi'(u_{x_k}^2)\, \Phi(u_{x_j}^2)\right)\,\chi\,dt\,dx
\\
&+8\left(\iint_{(T_0,\tau)\times \Omega}\langle D^2F(\nabla u)\nabla u_{x_j} , \nabla u_{x_j}\rangle\, u_{x_j}^2\,\Phi'(u_{x_j}^2)^2\, \Psi'(u_{x_k}^2)^\theta\, \chi\,\eta^2\,dt\,dx\right)^{\frac{1}{2}}\\
 & \times\Bigg(\frac{1}{4}\iint_{(T_0,\tau)\times \Omega} \chi' \eta^2 |u_{x_k}|^{2\theta}\, \Psi(u_{x_k}^2)^{2-\theta}\,dt\,dx
\\
&+\iint_{(T_0,\tau)\times \Omega}\langle D^2F(\nabla u) \nabla \eta, \nabla \eta\rangle\, |u_{x_k}|^{2\theta}\, \Psi(u_{x_k}^2)^{2-\theta}\,\chi\,dt\,dx\Bigg)^{\frac{1}{2}}.
\end{split}
\]
\end{lm}
\begin{proof}
It is convenient to divide the proof into two steps.
\vskip.2cm\noindent
{\bf Step 1: an identity involving \(u_{x_j}\) and \(u_{x_k}\)}. 
We first assume that $\Phi$ and $\Psi$ are two \(C^2\) non-decreasing and non-negative convex functions. We fix $k,j\in\{1,\dots,N\}$. Given $0<\sigma_1<(T_1-T_0)/2$  and $\zeta \in C^{\infty}_0(T_0+\sigma_1,T_1-\sigma_1)$, we consider \eqref{eq1043} with the index $j$ and for every \(0<\sigma<\sigma_1\), we insert the test function
\[
\varphi =\left(u_{\sigma}^+\right)_{x_j} \Phi'\left(\left(\left(u_{\sigma}^+\right)_{x_j}\right)^2\right)
\Psi\left(\left(\left(u_{\sigma}^+\right)_{x_k}\right)^2\right)\zeta \eta^2.
\]
Symmetrically, we consider \eqref{eq1043} with the index $k$ and insert the test function
\[
\widetilde{\varphi} =
\left(u_{\sigma}^+\right)_{x_k} \Psi'\left(\left(\left(u_{\sigma}^+\right)_{x_k}\right)^2\right)
\Phi\left(\left(\left(u_{\sigma}^+\right)_{x_j}\right)^2\right)\,\zeta\, \eta^2.
\]
The functions \(\varphi \) and \(\widetilde{\varphi}\) are compactly supported in \((T_0+\sigma_1,T_1-\sigma_1)\times\Omega\) and belong to \(L^{\infty}((T_0+\sigma_1,T_1-\sigma_1)\times\Omega)\cap L^{2}((T_0+\sigma_1,T_1-\sigma_1);W^{1,2}(\Omega))\). Thus they are admissible test functions.
We observe that
\[
\left((u_{\sigma}^+)_{x_j}\right)_{t}\, \varphi  
=
\frac{1}{2}\, \left(\Phi\left(\left(\left(u_{\sigma}^+\right)_{x_j}\right)^2\right) \right)_t\,\Psi\left(\left(\left(u_{\sigma}^+\right)_{x_k}\right)^2\right)\,\zeta\, \eta^2,
\]
and similarly
\[
\left((u_{\sigma}^+)_{x_k}\right)_{t}\, \widetilde{\varphi}  
=
\frac{1}{2} \,\left(\Psi\left(\left(\left(u_{\sigma}^+\right)_{x_k}\right)^2\right) \right)_t\,\Phi\left(\left(\left(u_{\sigma}^+\right)_{x_j}\right)^2\right)\,\zeta\, \eta^2.
\]
Thus we obtain
\begin{equation}
\label{doppietta}
\left((u_{\sigma}^+)_{x_j}\right)_{t} \varphi +\left((u_{\sigma}^+)_{x_k}\right)_{t} \widetilde{\varphi}  
=\frac{1}{2}\left[ \Phi\left(\left(\left(u_{\sigma}^+\right)_{x_j}\right)^2\right) \Psi\left(\left(\left(u_{\sigma}^+\right)_{x_k}\right)^2\right)\right]_t\,\zeta\, \eta^2.
\end{equation}
By summing the two equations obtained from \eqref{eq1043} as described above, and using the identity \eqref{doppietta}, we get
\[
\begin{split}
\frac{1}{2}\,\iint_{(T_0+\sigma_1,T_1-\sigma_1)\times\Omega}\Bigg[ \Phi\left(\left(\left(u_{\sigma}^+\right)_{x_j}\right)^2\right)& \Psi\left(\left(\left(u_{\sigma}^+\right)_{x_k}\right)^2\right)\Bigg]_t\,\zeta\, \eta^2\,dt\,dx\\
&+\iint_{(T_0+\sigma_1,T_1-\sigma_1)\times\Omega}\langle \left((\nabla F(\nabla u))_{x_j}\right)_{\sigma}^+ , \nabla \varphi\rangle\,dt\,dx\\
&+\iint_{(T_0+\sigma_1,T_1-\sigma_1)\times\Omega}\langle \left((\nabla F(\nabla u))_{x_k}\right)_{\sigma}^{+} , \nabla \widetilde{\varphi}\rangle\,dt\,dx=0.
\end{split}
\]
We then perform an integration by parts in the time variable in the first integral, which gives
\begin{multline*}
-\frac{1}{2}\iint_{(T_0+\sigma_1,T_1-\sigma_1)\times\Omega}\Phi\left(\left(\left(u_{\sigma}^+\right)_{x_j}\right)^2\right) \Psi\left(\left(\left(u_{\sigma}^+\right)_{x_k}\right)^2\right)\zeta' \eta^2\,dt\,dx \\
+\iint_{(T_0+\sigma_1,T_1-\sigma_1)\times\Omega}\langle \left((\nabla F(\nabla u))_{x_j}\right)_{\sigma}^+ , \nabla \varphi\rangle\,dt\,dx
+\iint_{(T_0+\sigma_1,T_1-\sigma_1)\times\Omega}\langle \left((\nabla F(\nabla u))_{x_k}\right)_{\sigma}^+ , \nabla \widetilde{\varphi}\rangle\,dt\,dx=0.
\end{multline*}
Finally, we let $\sigma$ go to $0$. In the same vein as in the proof of \eqref{eq110511}, the Dominated Convergence Theorem implies that
\begin{equation}
\begin{split}
-\frac{1}{2}&\iint_{I\times\Omega}\Phi\left(\left(u_{x_j}\right)^2\right) \Psi\left(\left(u_{x_k}\right)^2\right)\zeta' \eta^2\,dt\,dx \\
&+\iint_{I\times\Omega}\langle (\nabla F(\nabla u))_{x_j} , \nabla \varphi_0\rangle\,dt\,dx
+\iint_{I\times\Omega}\langle (\nabla F(\nabla u))_{x_k} , \nabla \widetilde{\varphi}_0\rangle\,dt\,dx=0, \label{eq1105}
\end{split}
\end{equation}
where
\[
\varphi_0 =u_{x_j} \Phi'\left(u_{x_j}^2\right)
\Psi\left(u_{x_k}^2\right)\zeta \eta^2,
\]
and 
\[
\widetilde{\varphi}_0 =
u_{x_k} \Psi'\left(u_{x_k}^2\right)
\Phi\left(u_{x_j}^2\right)\zeta \eta^2.
\]
We now choose \(\zeta\) as in \eqref{eq_zeta} and  insert it in \eqref{eq1105}. By letting \(\delta\) go to \(0\), we obtain for almost every \(\tau\in I\)
\begin{equation}
\begin{split}
\frac{\chi(\tau)}{2}&\int_{\{\tau\}\times\Omega}\Phi\left(u_{x_j}^2\right) \Psi\left(u_{x_k}^2\right) \eta^2\,dt\,dx \\
&+\iint_{(T_0,\tau)\times \Omega}\langle (\nabla F(\nabla u))_{x_j} , \nabla \psi_0\rangle\,dt\,dx
+\iint_{(T_0,\tau)\times \Omega}\langle (\nabla F(\nabla u))_{x_k} , \nabla \widetilde{\psi}_0\rangle\,dt\,dx\\
&= \frac{1}{2}\iint_{(T_0,\tau)\times \Omega} \chi' \eta^2 \Phi(u_{x_j}^2)\Psi(u_{x_k}^2)\,dt\,dx, \label{eq-1181}
\end{split}
\end{equation}
where \(\psi_0\) and \(\widetilde{\psi}_0\) are defined as \(\varphi_0\) and \(\widetilde{\varphi}_0\),  except that \(\zeta\) is now replaced by \(\chi\).
\vskip.2cm\noindent
{\bf Step 2: completion of the proof}.
We first observe that
\[
\langle (\nabla F(\nabla u))_{x_j} , \nabla \psi_0\rangle =\langle D^2F(\nabla u)\,\nabla u_{x_j}, \nabla \psi_0\rangle, \qquad \langle (\nabla F(\nabla u))_{x_k} , \nabla \widetilde{\psi}_0\rangle= \langle D^2F(\nabla u)
 \nabla u_{x_k}, \nabla \widetilde{\psi}_0\rangle.
\] 
Taking into account the definition of \(\psi_0\) and \(\widetilde{\psi}_0\), we thus get from \eqref{eq-1181}
\begin{equation}
\begin{split}
\frac{\chi(\tau)}{2}&\int_{\{\tau\}\times\Omega}\Phi\left(u_{x_j}^2\right) \Psi\left(u_{x_k}^2\right) \eta^2\,dx\\
&+\iint_{(T_0,\tau)\times \Omega}  \langle D^2F(\nabla u)\nabla u_{x_j} , \nabla u_{x_j}\rangle\left(\Phi'(u_{x_j}^2)+ 2 u_{x_j}^2\Phi''(u_{x_j}^2)\right) \Psi(u_{x_k}^2)\,\chi\,\eta^2 \,dt\,dx\\
&+\iint_{(T_0,\tau)\times \Omega}  \langle D^2F(\nabla u)\,\nabla u_{x_k}, \nabla u_{x_k}\rangle\left(\Psi'(u_{x_k}^2) + 2 u_{x_k}^2\Psi''(u_{x_k}^2)\right) \Phi(u_{x_j}^2)\,\chi\,\eta^2 \,dt\,dx
\\
&= \frac{1}{2}\iint_{(T_0,\tau)\times \Omega} \chi'\, \eta^2\, \Phi(u_{x_j}^2)\,\Psi(u_{x_k}^2)\,dt\,dx\\
&-4\, \iint_{(T_0,\tau)\times \Omega} \langle D^2F(\nabla u)\,\nabla u_{x_j} , \nabla u_{x_k}\rangle \,u_{x_j}\,\Phi'(u_{x_j}^2) \,u_{x_k}\,\Psi'(u_{x_k}^2)\,\chi\,\eta^2\,dt\,dx
\\
&-2 \iint_{(T_0,\tau)\times \Omega} \langle D^2F(\nabla u)\,\nabla u_{x_j} , \nabla \eta\rangle\, u_{x_j}\,\Phi'(u_{x_j}^2) \,\Psi(u_{x_k}^2)\,\chi\,\eta \,dt\,dx\\
&-2 \iint_{(T_0,\tau)\times \Omega} \langle D^2F(\nabla u)\,\nabla u_{x_k}, \nabla \eta\rangle\,u_{x_k}\,\Phi(u_{x_j}^2)\, \Psi'(u_{x_k}^2)\,\chi\,\eta \,dt\,dx.
\label{eq11534}
\end{split}
\end{equation}
We first estimate the two last terms. 
We use the Cauchy-Schwarz inequality:
\[
|\langle D^2F(\nabla u)\,\nabla u_{x_j} , \nabla \eta\rangle| \leq \bigg( \langle D^2F(\nabla u)\nabla u_{x_j} , \nabla u_{x_j}\rangle \bigg)^{\frac{1}{2}} \bigg(\langle D^2F(\nabla u)\nabla \eta , \nabla \eta\rangle\bigg)^{\frac{1}{2}}.
\]
By the Young inequality, this implies
\[
\begin{split}
\Bigg|-2 \iint_{(T_0,\tau)\times \Omega} \langle D^2F(\nabla u)\,\nabla u_{x_j} , \nabla \eta\rangle u_{x_j}\Phi'(u_{x_j}^2) \Psi(u_{x_k}^2)\,\chi\,\eta \,dt\,dx\Bigg|\\
\leq \frac{1}{2} \iint_{(T_0,\tau)\times \Omega}\langle D^2F(\nabla u)\nabla u_{x_j} , \nabla u_{x_j}\rangle\Phi'(u_{x_j}^2) \Psi(u_{x_k}^2)\,\chi\,\eta^2\,dt\,dx\\
+ 2 \, \iint_{(T_0,\tau)\times \Omega}\langle D^2F(\nabla u)\nabla \eta , \nabla \eta\rangle\,u_{x_j}^2\, \Phi'(u_{x_j}^2) \Psi(u_{x_k}^2)\,\chi\,dt\,dx.
\end{split}
\]
A similar inequality holds true for the last term in \eqref{eq11534}. Hence, 
\begin{equation}
\label{eq1153-bis}
\begin{split}
\frac{\chi(\tau)}{2}&\int_{\{\tau\}\times\Omega}\Phi\left(u_{x_j}^2\right) \Psi\left(u_{x_k}^2\right) \eta^2\,dx\\
&+\iint_{(T_0,\tau)\times \Omega}  \langle D^2F(\nabla u)\nabla u_{x_j} , \nabla u_{x_j}\rangle\left(\frac{1}{2}\Phi'(u_{x_j}^2) + 2\, u_{x_j}^2\,\Phi''(u_{x_j}^2)\right)\, \Psi(u_{x_k}^2)\,\chi\,\eta^2 \,dt\,dx\\
&+\iint_{(T_0,\tau)\times \Omega}  \langle D^2F(\nabla u)\,\nabla u_{x_k}, \nabla u_{x_k}\rangle\left(\frac{1}{2}\Psi'(u_{x_k}^2) + 2\, u_{x_k}^2\,\Psi''(u_{x_k}^2)\right)\, \Phi(u_{x_j}^2)\,\chi\,\eta^2 \,dt\,dx
\\
&\le \frac{1}{2}\iint_{(T_0,\tau)\times \Omega} \chi'\, \eta^2\, \Phi(u_{x_j}^2)\Psi(u_{x_k}^2)\,dt\,dx\\
&-4 \iint_{(T_0,\tau)\times \Omega} \langle D^2F(\nabla u)\,\nabla u_{x_j} , \nabla u_{x_k}\rangle u_{x_j}\,\Phi'(u_{x_j}^2)\, u_{x_k}\Psi'(u_{x_k}^2)\,\chi\,\eta^2\,dt\,dx
\\
&+2 \iint_{(T_0,\tau)\times \Omega} \langle D^2F(\nabla u)\nabla \eta , \nabla \eta\rangle\,\left(u_{x_j}^2\,\Phi'(u_{x_j}^2) \,\Psi(u_{x_k}^2)+u_{x_k}^2\Psi'(u_{x_k}^2) \Phi(u_{x_j}^2)\right)\chi\,dt\,dx.
\end{split}
\end{equation}
In the left-hand side of \eqref{eq1153-bis}, in the second term, we drop \(2\,u_{x_j}^2\,\Phi''(u_{x_j}^2)\) which is non-negative. We also drop the whole last term of the left-hand side for the same reason. 
This yields
\begin{equation}
\begin{split}
\frac{\chi(\tau)}{2}&
\int_{\{\tau\}\times\Omega}\Phi\left(u_{x_j}^2\right) \Psi\left(u_{x_k}^2\right) \eta^2\,dx\\
& +\frac{1}{2}\iint_{(T_0,\tau)\times \Omega} \langle D^2F(\nabla u)\nabla u_{x_j} , \nabla u_{x_j}\rangle \Phi'(u_{x_j}^2)  \Psi(u_{x_k}^2)\chi\eta^2 \,dt\,dx\\
&\leq \frac{1}{2}\iint_{(T_0,\tau)\times \Omega} \chi' \eta^2 \Phi(u_{x_j}^2)\Psi(u_{x_k}^2)\,dt\,dx\\
&-4\, \iint_{(T_0,\tau)\times \Omega} \langle D^2F(\nabla u)\,\nabla u_{x_j} , \nabla u_{x_k}\rangle\, u_{x_j}\,\Phi'(u_{x_j}^2) \,u_{x_k}\,\Psi'(u_{x_k}^2)\,\chi\,\eta^2\,dt\,dx
\\
&+2 \iint_{(T_0,\tau)\times \Omega} \langle D^2F(\nabla u)\nabla \eta , \nabla \eta\rangle \left(u_{x_j}^2\Phi'(u_{x_j}^2) \Psi(u_{x_k}^2)+u_{x_k}^2\Psi'(u_{x_k}^2) \Phi(u_{x_j}^2)\right)\chi\,dt\,dx.
\label{eq1178}
\end{split}
\end{equation}
We next estimate the second term of the right-hand side of \eqref{eq1178}, that we denote by 
\[
\mathcal{A}= \iint_{(T_0,\tau)\times \Omega} \langle D^2F(\nabla u)\,\nabla u_{x_j} , \nabla u_{x_k}\rangle\, u_{x_j}\,\Phi'(u_{x_j}^2) \,u_{x_k}\,\Psi'(u_{x_k}^2)\,\chi\,\eta^2\,dt\,dx.
\]
We first use the Cauchy-Schwarz inequality to get
\[
|\langle D^2F(\nabla u)\,\nabla u_{x_j} , \nabla u_{x_k}\rangle|  \leq 
\bigg( \langle D^2F(\nabla u)\nabla u_{x_j} , \nabla u_{x_j}\rangle \bigg)^{\frac{1}{2}} \bigg( \langle D^2F(\nabla u)\,\nabla u_{x_k}, \nabla u_{x_k}\rangle\bigg)^{\frac{1}{2}}.
\]
We then introduce a parameter \(\theta\in [0,1]\). By writing \(\Psi'(u_{x_k}^2) = \Psi'(u_{x_k}^2)^{\frac{\theta}{2}}\,\Psi'(u_{x_k}^2)^{1-\frac{\theta}{2}}\), one gets by the Cauchy-Schwarz inequality again:
\begin{multline}\label{eq157}
\mathcal{A} \leq 
\left(\iint_{(T_0,\tau)\times \Omega}\langle D^2F(\nabla u)\nabla u_{x_j} , \nabla u_{x_j}\rangle\, u_{x_j}^2\,\Phi'(u_{x_j}^2)^2 \,\Psi'(u_{x_k}^2)^\theta\, \chi\,\eta^2\,dt\,dx\right)^{\frac{1}{2}}\\
\left(\iint_{(T_0,\tau)\times \Omega} \langle D^2F(\nabla u)\,\nabla u_{x_k}, \nabla u_{x_k}\rangle\, \Psi'(u_{x_k}^2)^{2-\theta}\,u_{x_k}^2\,\chi\,\eta^2\,dt\,dx\right)^{\frac{1}{2}}.
\end{multline}
We define
\[
G(t)=\int_{0}^{t^2}\Psi'(s)^{1-\frac{\theta}{2}}\,ds.
\]
Then \(G\) is a \(C^1\) non-negative convex function. Moreover, by its definition
\[
\nabla (G\circ u_{x_k}) = 2\, u_{x_k}\, \Psi'(u_{x_k}^2)^{1-\frac{\theta}{2}}\,\nabla u_{x_k}.
\]
Hence, by  the standard Caccioppoli inequality \eqref{eq1218} with \(h=G\) and \(k\) in place of \(j\), this yields:
\[
\begin{split}
\iint_{(T_0,\tau)\times \Omega} &\langle D^2F(\nabla u)\,\nabla u_{x_k}, \nabla u_{x_k}\rangle \,\Psi'(u_{x_k}^2)^{2-\theta}u_{x_k}^2\,\chi\,\eta^2\,dt\,dx\\
&= \frac{1}{4}\iint_{(T_0,\tau)\times \Omega}\langle D^2F(\nabla u)\nabla G(u_{x_k}) , \nabla G(u_{x_k})\rangle\,\chi\,\eta^2\,dt\,dx\\
&\leq \frac{1}{4}\iint_{(T_0,\tau)\times \Omega} \chi'\, \eta^2\, G^2(u_{x_k})\,dt\,dx
+\iint_{(T_0,\tau)\times \Omega}\langle D^2F(\nabla u)\nabla \eta , \nabla \eta\rangle\, G^2(u_{x_k})\,\chi\,dt\,dx.
\end{split}
\]
By using the Jensen inequality with the concave function $y\mapsto y^{1-\theta/2}$, we obtain
\[
0\leq G(u_{x_k}) \leq |u_{x_k}|^{\theta}\, \Psi(u_{x_k}^2)^{1-\frac{\theta}{2}}.
\]
This implies
\[
\begin{split}
\iint_{(T_0,\tau)\times \Omega} &\langle D^2F(\nabla u)\,\nabla u_{x_k}, \nabla u_{x_k}\rangle \Psi'(u_{x_k}^2)^{2-\theta}\,u_{x_k}^2\,\chi\,\eta^2\,dt\,dx\\
&\leq \frac{1}{4}\iint_{(T_0,\tau)\times \Omega} \chi'\, \eta^2\, |u_{x_k}|^{2\theta}\, \Psi(u_{x_k}^2)^{2-\theta}\,dt\,dx \\
&+ \iint_{(T_0,\tau)\times \Omega} \langle D^2F(\nabla u) \nabla \eta, \nabla \eta\rangle\, |u_{x_k}|^{2\theta}\, \Psi(u_{x_k}^2)^{2-\theta}\,\chi\,dt\,dx.
\end{split}
\]
Coming back to \eqref{eq157}, it follows that
\[
\begin{split}
\mathcal{A} &\leq 
\left(\iint_{(T_0,\tau)\times \Omega} \langle D^2F(\nabla u)\nabla u_{x_j} , \nabla u_{x_j}\rangle\, u_{x_j}^2\,\Phi'(u_{x_j}^2)^2\, \Psi'(u_{x_k}^2)^\theta\, \chi\,\eta^2\,dt\,dx\right)^{\frac{1}{2}}\\
&\times\left(\frac{1}{4}\iint_{(T_0,\tau)\times \Omega} \chi'\, \eta^2\, |u_{x_k}|^{2\theta}\, \Psi(u_{x_k}^2)^{2-\theta}\,dt\,dx\right.\\
&\left.+\iint_{(T_0,\tau)\times \Omega} \langle D^2F(\nabla u) \nabla \eta, \nabla \eta\rangle\, |u_{x_k}|^{2\theta}\, \Psi(u_{x_k}^2)^{2-\theta}\,\chi\,dt\,dx\right)^{\frac{1}{2}}.
\end{split}
\]
Together with \eqref{eq1178}, this yields the desired inequality.
Finally, the $C^2$ assumption on $\Phi$ and $\Psi$ can be removed by a standard approximation argument.
\end{proof}

\section{Uniform Lipschitz estimate for the regularized equation} 
\label{sec:4}

This section is devoted to the proof of the following uniform estimate. For simplicity, we will work with anisotropic parabolic cubes of the form
\[
Q_R(t_0,x_0)=(t_0-R^p,t_0)\times (x_0-R,x_0+R)^N.
\]
\begin{prop}\label{lm-apriori-estimate}
There exist $\alpha=\alpha(N)>2$ and \(C=C(N,p)>0\) such that for every \(\varepsilon>0\) and for every \(Q_{r}(x_0,t_0) \subset Q_R(t_0,x_0)\Subset I\times\Omega\) with \(R\leq 1\), one has 
\begin{equation}
\label{uepsbound}
\|\nabla u_\varepsilon\|_{L^{\infty}(Q_{r}(x_0,t_0))} \leq \frac{C}{(R-r)^{\alpha\,p}}\,\left[\left(\iint_{Q_R(t_0,x_0)} |\nabla u_{\varepsilon}|^p\,dt\,dx\right)^\frac{1}{2}+1\right].
\end{equation}
\end{prop}
\begin{proof} 
We will limit ourselves for simplicity to the case \(N\geq 3\). This allows to use the Sobolev inequality valid for every $f\in W^{1,2}_0(\Omega)$
\[
\|f\|_{L^{2^*}(\Omega)}\leq C_N\, \|\nabla f\|_{L^{2}(\Omega)},\qquad \mbox{ with }2^{*}=\frac{2\,N}{N-2}.
\] 
Here \(C_N\) is a constant which depends only on \(N\). The case \(N=2\) follows with minor modifications and we omit the details.
\par
The proof is quite involved and for ease of readability, we divide it into several steps.
\vskip.2cm\noindent
{\bf Step 1: the choices of \(\Phi\) and \(\Psi\)}. 
We apply Lemma \ref{lm:weird} with the following choices
\[
\Phi(t)=t^{s} \qquad \textrm{ and } \qquad \Psi(t)=t^m, \qquad \textrm{ for } t\geq 0,
\]
with \(1\leq s \leq m\). We also take
\[
\theta = \begin{cases}
\dfrac{m-s}{m-1} \in [0,1] &\textrm{ if } m>1,\\
1 & \textrm{ if } m=1.
\end{cases}
\]
This gives
\[
\begin{split}
&\chi(\tau)
\int_{\{\tau\}\times\Omega}|u_{x_j}|^{2s}\, |u_{x_k}|^{2m}\, \eta^2\,dx 
\\
&+s\iint_{(T_0,\tau)\times \Omega} \langle D^2F(\nabla u)\nabla u_{x_j} , \nabla u_{x_j}\rangle\, |u_{x_j}|^{2s-2}\,  |u_{x_k}|^{2m}\chi\,\eta^2 \,dt\,dx \\
&\leq \iint_{(T_0,\tau)\times \Omega} \chi'\, \eta^2\, |u_{x_j}|^{2s}\,|u_{x_k}|^{2m}\,dt\,dx\\
&+4\iint_{(T_0,\tau)\times \Omega} \langle D^2F(\nabla u)\nabla \eta , \nabla \eta\rangle\left(s\, |u_{x_j}|^{2s}\,|u_{x_k}|^{2m}+m\, |u_{x_k}|^{2m}\, |u_{x_j}|^{2s}\right)\chi\,dt\,dx
\\
&+8\,s\,m^{\frac{\theta}{2}}\left(\iint_{(T_0,\tau)\times \Omega}\langle D^2F(\nabla u)\nabla u_{x_j} , \nabla u_{x_j}\rangle |u_{x_j}|^{4s-2} |u_{x_k}|^{2m-2s} \chi\eta^2\,dt\,dx\right)^{\frac{1}{2}}\\
 & \qquad \times\left(\frac{1}{4}\iint_{(T_0,\tau)\times \Omega} \chi' \eta^2 |u_{x_k}|^{2(s+m)} \,dt\,dx
+\iint_{(T_0,\tau)\times \Omega} \langle D^2F(\nabla u) \nabla \eta, \nabla \eta\rangle |u_{x_k}|^{2(s+m)} \chi\,dt\,dx\right)^{\frac{1}{2}}.
\end{split}
\]
On the product of the two last integrals, we use the Young inequality in the form
\[
a\,b\le a^2+\frac{b^2}{4},
\]
this gives
\begin{align*}
&\chi(\tau)
\int_{\{\tau\}\times\Omega}|u_{x_j}|^{2s}\, |u_{x_k}|^{2m}\, \eta^2\,dx 
+s\,\iint_{(T_0,\tau)\times \Omega} \langle D^2F(\nabla u)\nabla u_{x_j} , \nabla u_{x_j}\rangle\, |u_{x_j}|^{2s-2}\,  |u_{x_k}|^{2m}\,\chi\,\eta^2 \,dt\,dx \\
&\leq \iint_{(T_0,\tau)\times \Omega} \chi'\, \eta^2 |u_{x_j}|^{2s}\,|u_{x_k}|^{2m}\,dt\,dx+4\,(s+m)\,\iint_{(T_0,\tau)\times \Omega} \langle D^2F(\nabla u)\nabla \eta , \nabla \eta\rangle\, |u_{x_j}|^{2s}\,|u_{x_k}|^{2m}\,\chi\,dt\,dx
\\
&+16\,s^2\, m^\theta\left(\frac{1}{4}\iint_{(T_0,\tau)\times \Omega} \chi'\, \eta^2\, |u_{x_k}|^{2(s+m)} \,dt\,dx
+\iint_{(T_0,\tau)\times \Omega} \langle D^2F(\nabla u) \nabla \eta, \nabla \eta\rangle\, |u_{x_k}|^{2(s+m)}\, \chi\,dt\,dx\right)\\
&+\iint_{(T_0,\tau)\times \Omega}\langle D^2F(\nabla u)\nabla u_{x_j} , \nabla u_{x_j}\rangle\, |u_{x_j}|^{4s-2}\, |u_{x_k}|^{2m-2s}\, \chi\,\eta^2\,dt\,dx.
\end{align*}
By the Young inequality again, we can estimate 
\[
|u_{x_j}|^{2s}\, |u_{x_k}|^{2m}\leq |u_{x_j}|^{2m+2s} + |u_{x_k}|^{2m+2s}.
\]
Using that \(s\geq 1\) in the left-hand side and \(m^\theta\leq m\) in the right-hand side, we thus obtain
\begin{align*}
&\chi(\tau)
\int_{\{\tau\}\times\Omega}|u_{x_j}|^{2s}\, |u_{x_k}|^{2m}\, \eta^2\,dx 
+\iint_{(T_0,\tau)\times \Omega}  \langle D^2F(\nabla u)\nabla u_{x_j} , \nabla u_{x_j}\rangle\, |u_{x_j}|^{2s-2}\,  |u_{x_k}|^{2m}\,\chi\,\eta^2 \,dt\,dx \\
&\leq \iint_{(T_0,\tau)\times \Omega} \chi'\, \eta^2\, (|u_{x_j}|^{2(s+m)}+|u_{x_k}|^{2(s+m)})\,dt\,dx\\
&+4\,(s+m)\,\iint_{(T_0,\tau)\times \Omega} \langle D^2F(\nabla u)\nabla \eta , \nabla \eta\rangle\, \left(|u_{x_j}|^{2(s+m)}+|u_{x_k}|^{2(s+m)}\right)\,\chi\,dt\,dx
\\
&+16\,s^2\, m\left(\frac{1}{4}\iint_{(T_0,\tau)\times \Omega} \chi'\, \eta^2\, |u_{x_k}|^{2(s+m)} \,dt\,dx
+\iint_{(T_0,\tau)\times \Omega} \langle D^2F(\nabla u) \nabla \eta, \nabla \eta\rangle |u_{x_k}|^{2(s+m)}\, \chi\,dt\,dx\right)\\
&+\iint_{(T_0,\tau)\times \Omega} \langle D^2F(\nabla u)\nabla u_{x_j} , \nabla u_{x_j}\rangle |u_{x_j}|^{4s-2}\, |u_{x_k}|^{2m-2s}\, \chi\,\eta^2\,dt\,dx.
\end{align*}
This finally implies
\begin{equation}
\label{presta}
\begin{split}
&\chi(\tau)
\int_{\{\tau\}\times\Omega}|u_{x_j}|^{2s}\, |u_{x_k}|^{2m}\, \eta^2\,dx 
+\iint_{(T_0,\tau)\times \Omega}  \langle D^2F(\nabla u)\nabla u_{x_j} , \nabla u_{x_j}\rangle\, |u_{x_j}|^{2s-2}\,  |u_{x_k}|^{2m}\,\chi\,\eta^2 \,dt\,dx \\
&\leq 16\,(s+m +s^2\,m)\iint_{(T_0,\tau)\times \Omega} \left(\chi' \,\eta^2 +\chi\, \langle D^2F(\nabla u)\nabla \eta , \nabla \eta\rangle\right)\,\left(|u_{x_j}|^{2(s+m)}+|u_{x_k}|^{2(s+m)}\right)\,dt\,dx\\
&+\iint_{(T_0,\tau)\times \Omega}\langle D^2 F(\nabla u)\nabla u_{x_j} , \nabla u_{x_j}\rangle\, |u_{x_j}|^{4s-2} |u_{x_k}|^{2m-2s} \,\chi\,\eta^2\,dt\,dx.
\end{split}
\end{equation}
{\bf Step 2: the staircase}.
Let \(\ell_0\in \mathbb{N}\setminus \{0\}\) and set \(q=2^{\ell_0}-1\). We define the two families of indices 
\[
s_\ell = 2^\ell \qquad \mbox{ and }\qquad m_\ell = q+1-2^\ell, \qquad \mbox{ for } \ell \in \{0, \dots, \ell_0\}.
\]
We observe that by construction, for every \(0\leq \ell\leq \ell_0-1\), 
\[
s_\ell+m_\ell = q+1,\qquad 4s_{\ell}-2 = 2s_{\ell+1}-2 \qquad \mbox{ and }\qquad 2m_\ell-2s_\ell=2m_{\ell+1}.
\] 
We also use that \(s_\ell+m_\ell +s_{\ell}^2\,m_\ell \leq 2\,(q+1)^3\).
Then the above inequality \eqref{presta} written for \(s=s_\ell\) and \(m=m_\ell\) with \(0\leq \ell\leq \ell_0-1\) gives
\[
\begin{split}
\chi(\tau)\,&
\int_{\{\tau\}\times\Omega}|u_{x_j}|^{2s_\ell}\, |u_{x_k}|^{2m_\ell}\, \eta^2\,dx 
\\
&+\iint_{(T_0,\tau)\times \Omega} \langle D^2F(\nabla u)\nabla u_{x_j} , \nabla u_{x_j}\rangle\, |u_{x_j}|^{2s_\ell-2}\,  |u_{x_k}|^{2m_\ell}\,\chi\,\eta^2 \,dt\,dx \\
&\leq 32\,(q+1)^3\,\iint_{(T_0,\tau)\times \Omega} \left(\chi'\, \eta^2 +\chi\,\langle D^2F(\nabla u)\nabla \eta , \nabla \eta\rangle\right)\left(|u_{x_j}|^{2(q+1)}+|u_{x_k}|^{2(q+1)}\right)\,dt\,dx\\
&+\iint_{(T_0,\tau)\times \Omega} \langle D^2F(\nabla u)\nabla u_{x_j} , \nabla u_{x_j}\rangle\, |u_{x_j}|^{2s_{\ell+1}-2}\, |u_{x_k}|^{2m_{\ell+1}}\, \chi\,\eta^2\,dt\,dx.
\end{split}
\]
Observe that we used that $2\,s_\ell+2\,m_\ell=2\,(q+1)$ on the first term on the right-hand side.
By summing from \(\ell=0\) up to \(\ell=\ell_0-1\), one gets
\begin{align*}
&\chi(\tau)
\sum_{\ell=0}^{\ell_0-1}\int_{\{\tau\}\times\Omega}|u_{x_j}|^{2s_\ell}\, |u_{x_k}|^{2m_\ell} \,\eta^2\,dx 
+\iint_{(T_0,\tau)\times \Omega}  \langle D^2F(\nabla u)\nabla u_{x_j} , \nabla u_{x_j}\rangle\, |u_{x_k}|^{2q}\,\chi\,\eta^2 \,dt\,dx \\
&\leq C\,q^3\,\ell_0\iint_{(T_0,\tau)\times \Omega} \left(\chi' \eta^2 +\chi\,\langle D^2F(\nabla u)\nabla \eta , \nabla \eta\rangle\right)\left(|u_{x_j}|^{2(q+1)}+|u_{x_k}|^{2(q+1)}\right)\,dt\,dx\\
&+\iint_{(T_0,\tau)\times \Omega} \langle D^2F(\nabla u)\nabla u_{x_j} , \nabla u_{x_j}\rangle\, |u_{x_j}|^{2q}\, \chi\,\eta^2\,dt\,dx.
\end{align*}
For the last term, we apply Lemma \ref{lm:standard} with the choice 
\[
h(t)=\frac{|t|^{q+1}}{q+1}, \qquad t\in \mathbb{R}.
\]
We thus get
\[
\begin{split}
\chi(\tau)\,
\sum_{\ell=0}^{\ell_0-1}\int_{\{\tau\}\times\Omega}|u_{x_j}|^{2s_\ell}\,& |u_{x_k}|^{2m_\ell}\, \eta^2\,dx 
+\iint_{(T_0,\tau)\times \Omega}  \langle D^2F(\nabla u)\,\nabla u_{x_j} , \nabla u_{x_j}\rangle\,|u_{x_k}|^{2q}\,\chi\,\eta^2 \,dt\,dx \\
&\leq C\,q^4\,\iint_{(T_0,\tau)\times \Omega} \left(\chi'\, \eta^2 +\chi\,\langle D^2F(\nabla u)\nabla \eta , \nabla \eta\rangle\right)\,\left(|u_{x_j}|^{2\,(q+1)}+|u_{x_k}|^{2\,(q+1)}\right)\,dt\,dx\\
&+ \frac{1}{(q+1)^2}\iint_{(T_0,\tau)\times \Omega} \chi'\, \eta^2\, |u_{x_j}|^{2(q+1)}\,dt\,dx\\
&+\frac{4}{(q+1)^2}\iint_{(T_0,\tau)\times \Omega}\langle D^2F(\nabla u)\nabla \eta, \nabla \eta\rangle\, |u_{x_j}|^{2(q+1)}\,\chi \,dt\,dx.
\end{split}
\]
This is turn implies that
\begin{equation}
\begin{split}
&\chi(\tau)
\sum_{\ell=0}^{\ell_0-1}\int_{\{\tau\}\times\Omega}|u_{x_j}|^{2s_\ell} |u_{x_k}|^{2m_\ell} \eta^2\,dx 
+\iint_{(T_0,\tau)\times \Omega}  \langle D^2F(\nabla u)\nabla u_{x_j} , \nabla u_{x_j}\rangle |u_{x_k}|^{2q}\chi\eta^2 \,dt\,dx\\
&\leq C\,q^4\,\iint_{(T_0,\tau)\times \Omega} \left(\chi' \eta^2 +\chi\,\langle D^2F(\nabla u)\nabla \eta , \nabla \eta\rangle\right)\left(|u_{x_j}|^{2(q+1)}+|u_{x_k}|^{2(q+1)}\right)\,dt\,dx, \label{eq1443}
\end{split}
\end{equation}
possibly for a different constant $C>0$.
\vskip.2cm\noindent
{\bf Step 3: weak ellipticity and boundedness of \(D^2F\)}.
We now use the explicit expression of \(F\). We recall that
\[
F(\xi)=\frac{1}{p}\,\sum_{i=1}^N |\xi_i|^p + \varepsilon\, G(\xi),\qquad \mbox{ for every } \xi \in\mathbb{R}^N,
\]
where 
\[
G(\xi)=\frac{1}{p}\,\left(1+|\xi|^2\right)^{\frac{p}{2}},\qquad \mbox{ for every }\xi \in\mathbb{R}^N.
\]
If follows that for every $\xi,\,\lambda\in\mathbb{R}^N$, we have
\[
(p-1)\,\sum_{i=1}^{N}|\xi_i|^{p-2}\,\lambda_{i}^2\le \langle D^2F(\xi)\lambda, \lambda\rangle \leq C\,\left(|\xi|^{p-2} +1\right)\,|\lambda|^2.
\]
By inserting these estimates into \eqref{eq1443}, one gets
\[
\begin{split}
\chi(\tau)\,
\sum_{\ell=0}^{\ell_0-1}&\int_{\{\tau\}\times\Omega}|u_{x_j}|^{2s_\ell}\, |u_{x_k}|^{2m_\ell}\, \eta^2\,dx 
\\
&+(p-1)\,\sum_{i=1}^N\iint_{(T_0,\tau)\times \Omega}  |u_{x_i}|^{p-2}\, u_{x_ix_j}^2\, |u_{x_k}|^{2q}\,\chi\,\eta^2 \,dt\,dx \\
&\leq C\,q^4\,\iint_{(T_0,\tau)\times \Omega} \left(\chi'\, \eta^2 +\chi\,\left(|\nabla u|^{p-2}+1\right)\,|\nabla \eta|^2\right)\,\left(|u_{x_j}|^{2(q+1)}+|u_{x_k}|^{2(q+1)}\right)\,dt\,dx. 
\end{split}
\]
We consider the second term on the left-hand side: observe that by keeping in the sum only the term with $i=k$ and dropping the others, we get
\[
\begin{split}
\sum_{i=1}^N \iint_{(T_0,\tau)\times \Omega} |u_{x_i}|^{p-2}\,u_{x_ix_j}^2\, |u_{x_k}|^{2q}\,\chi\,\eta^2 \,dt\,dx
 & \geq  \iint_{(T_0,\tau)\times \Omega} |u_{x_k}|^{p-2}\,u_{x_kx_j}^2\, |u_{x_k}|^{2q}\,\chi\,\eta^2 \,dt\,dx\\
&= \frac{1}{\left(q+\dfrac{p}{2}\right)^2}\iint_{(T_0,\tau)\times \Omega} \left|\left( |u_{x_k}|^{q+\frac{p-2}{2}}u_{x_k}\right)_{x_j}\right|^2\,\chi\,\eta^2 \,dt\,dx.
\end{split}
\]
When we sum over $j=1,\dots,N$ the resulting estimate, we thus get
\[
\begin{split}
\chi(\tau)\,
\sum_{\ell=0}^{\ell_0-1}&\int_{\{\tau\}\times\Omega}\sum_{j=1}^{N}|u_{x_j}|^{2s_\ell}\, |u_{x_k}|^{2m_\ell} \,\eta^2\,dx 
+\frac{p-1}{\left(q+\dfrac{p}{2}\right)^2}\iint_{(T_0,\tau)\times \Omega}\left|\nabla \left( |u_{x_k}|^{q+\frac{p-2}{2}}\,u_{x_k}\right)\right|^2\, \chi\,\eta^2 \,dt\,dx \\
&\leq Cq^4\iint_{(T_0,\tau)\times \Omega} \left(\chi' \eta^2 +\chi\left(|\nabla u|^{p-2}+1\right)|\nabla \eta|^2\right)\left(\sum_{j=1}^{N} |u_{x_j}|^{2(q+1)}+N\,|u_{x_k}|^{2(q+1)}\right)\,dt\,dx,
\end{split}
\]
which in turn implies
\[
\begin{split}
&\chi(\tau)
\sum_{\ell=0}^{\ell_0-1}\int_{\{\tau\}\times\Omega}\sum_{j=1}^{N}|u_{x_j}|^{2s_\ell}\, |u_{x_k}|^{2m_\ell} \,\eta^2\,dx 
+\iint_{(T_0,\tau)\times \Omega}\left|\nabla \left( |u_{x_k}|^{q+\frac{p-2}{2}}u_{x_k}\right)\right|^2 \,\chi\,\eta^2 \,dt\,dx \\
&\leq Cq^6\iint_{(T_0,\tau)\times \Omega} \left(\chi' \eta^2 +\chi\left(|\nabla u|^{p-2}+1\right)|\nabla \eta|^2\right)\,|\nabla u|^{2(q+1)}\,dt\,dx,
\end{split}
\]
up to redefine the constant $C>0$.
We now add the term 
\[
\iint_{(T_0,\tau)\times \Omega} |u_{x_k}|^{2q+p}\, \chi\,|\nabla \eta|^2 \,dt\,dx,
\] 
on both sides of the above inequality. With some algebraic manipulations, this gives
\begin{align*}
&\chi(\tau)
\sum_{\ell=0}^{\ell_0-1}\int_{\{\tau\}\times\Omega}\sum_{j=1}^{N}|u_{x_j}|^{2s_\ell} \,|u_{x_k}|^{2m_\ell}\, \eta^2\,dx 
+\iint_{(T_0,\tau)\times \Omega}\left|\nabla \left( |u_{x_k}|^{q+\frac{p-2}{2}}\,u_{x_k}\,\eta\right)\right|^2\, \chi \,dt\,dx \\
&\leq C\,q^6\iint_{(T_0,\tau)\times \Omega} \left(\chi'\, \eta^2 +\chi\,\left(|\nabla u|^{p-2}+1\right)|\nabla \eta|^2\right)\,|\nabla u|^{2(q+1)}\,dt\,dx
+\iint_{(T_0,\tau)\times \Omega} |u_{x_k}|^{2q+p} \chi|\nabla \eta|^2 \,dt\,dx\\
&\leq C\,q^6\iint_{(T_0,\tau)\times \Omega} \left(\chi' \eta^2 +\chi\left(|\nabla u|^{p-2}+1\right)|\nabla \eta|^2\right)|\nabla u|^{2(q+1)}\,dt\,dx.
\end{align*}
By using the Sobolev inequality in the spatial variable for the second term of the left-hand side, one obtains
\[
\begin{split}
&\chi(\tau)
\sum_{\ell=0}^{\ell_0-1}\int_{\{\tau\}\times\Omega}\sum_{j=1}^{N}|u_{x_j}|^{2s_\ell}\, |u_{x_k}|^{2m_\ell}\, \eta^2\,dx 
+\int_{T_0}^{\tau}\chi\left(\int_{\Omega}\left( |u_{x_k}|^{2q+p}\eta^2\right)^\frac{2^*}{2}\,dx\right)^{\frac{2}{2^*}}\,dt \\
&\leq C\,q^6\iint_{(T_0,\tau)\times \Omega} \left(\chi'\, \eta^2 +\chi\,\left(|\nabla u|^{p-2}+1\right)|\nabla \eta|^2\right)\,|\nabla u|^{2(q+1)}\,dt\,dx.
\end{split}
\]
We now finally sum over \(k=1, \dots, N\) and use the Minkowski inequality for the second term of the left-hand side. This gives
\[
\begin{split}
&\chi(\tau)
\sum_{\ell=0}^{\ell_0-1}\int_{\{\tau\}\times\Omega}\sum_{j=1}^{N}|u_{x_j}|^{2s_\ell} \sum_{k=1}^{N}|u_{x_k}|^{2m_\ell}\, \eta^2\,dx 
+\int_{T_0}^{\tau}\chi\left(\int_{\Omega}\left( \sum_{k=1}^{N}|u_{x_k}|^{2q+p}\eta^2\right)^\frac{2^*}{2}\,dx\right)^{\frac{2}{2^*}}\,dt \\
&\leq C\,q^6\iint_{(T_0,\tau)\times \Omega} \left(\chi' \eta^2 +\chi\left(|\nabla u|^{p-2}+1\right)|\nabla \eta|^2\right)|\nabla u|^{2(q+1)}\,dt\,dx.
\end{split}
\]
We introduce the expedient function 
\[
\mathcal{U}=\max_{1\leq k \leq N} |u_{x_k}|.
\] 
We observe that for every \(s\geq 0\), one has 
\[
\mathcal{U}^{s}\leq \sum_{k=1}^N|u_{x_k}|^{s}\leq N\, \mathcal{U}^{s}.
\] 
In particular for \(s=2\), we get \(\mathcal{U}\leq |\nabla u|\leq \sqrt{N}\, \mathcal{U}\). Hence, we get
\begin{equation}
\begin{split}
\chi(\tau)
\int_{\{\tau\}\times\Omega}\mathcal{U}^{2(q+1)}\, \eta^2\,dx& 
+\int_{T_0}^{\tau}\chi\left(\int_{\Omega} \mathcal{U}^{(2q+p)\frac{2^*}{2}}\,\eta^{2^*}\,dx\right)^{\frac{2}{2^*}}dt \\
&\leq C\,q^6\iint_{(T_0,\tau)\times \Omega} \left(\chi'\, \eta^2 +\chi\,\left(\mathcal{U}^{p-2}+1\right)\,|\nabla \eta|^2\right)\mathcal{U}^{2(q+1)}\,dt\,dx\\
&\leq C\,q^6\iint_{(T_0,\tau)\times \Omega} \left(\chi'\, \eta^2 +\chi\,|\nabla \eta|^2\right)\left(1+\mathcal{U}^{2q+p}\right)\,dt\,dx.\label{eq1529}
\end{split}
\end{equation}
{\bf Step 4: choice of the cut-off functions}.
 Let \((t_0,x_0)\in I\times\Omega\) and \(0<r<R\le 1\) such that the cube \(Q_{R}(x_0)=(x_0-R,x_0+R)^N\) is compactly contained in \(\Omega\). We further require that 
\[
(t_0-R^p,t_0)\Subset I,
\]
so that we must have $T_0<t_0<T_1$ and \(R^p<t_0-T_0\). Let \(\chi:[T_0,T_1]\to \mathbb{R} \) be a non-decreasing Lipschitz function such that 
 \[
\chi(t)=0\ \mbox{ on } [T_0, t_0-R^p],\qquad \chi(t)=1\ \mbox{ on } [t_0-r^p, t_0]\qquad \mbox{ and }\qquad |\chi'(t)|\leq \frac{C}{(R-r)^p}.
\]
Let \(\eta\in C^{\infty}_0(Q_R(x_0))\) be such that 
\[
0\leq \eta\leq 1,\qquad  \eta=1\ \mbox{ on } Q_{r}(x_0)\qquad \mbox{ and }\qquad |\nabla \eta|\leq \frac{C}{R-r}.
\] 
We recall the notation for the anisotropic parabolic cube 
\[
Q_\rho(t_0,x_0)=(t_0-\rho^p, t_0)\times Q_{\rho}(x_0) .
\] 
With such a choice of \(\chi\) and $\eta$, we use \eqref{eq1529} twice: 
\begin{itemize}
\item firstly, by dropping the second term in the left-hand side, and taking the supremum in $\tau$ over the interval $(t_{0}-r^p,t_0)$; 
\vskip.2cm
\item secondly, by dropping the first term in the left-hand side and taking \(\tau = t_0\). 
\end{itemize}
By summing up the two resulting contributions, this gives
\begin{align}
\sup_{\tau\in (t_0-r^p, t_0)}
\int_{\{\tau\}\times Q_r(x_0)}\mathcal{U}^{2(q+1)} \,dx& 
+\int_{t_0-r^p}^{t_0}\,\left(\int_{Q_r(x_0)} \mathcal{U}^{(2q+p)\frac{2^*}{2}}\,dx\right)^{\frac{2}{2^*}}\,dt \nonumber\\
&\leq C\frac{q^6}{(R-r)^p}\,\int_{Q_R(t_0,x_0)} \left(1+\mathcal{U}^{2q+p}\right)\,dt\,dx. \label{eq1551}
\end{align}
Observe that we also used that $(R-r)^p\le (R-r)^2$, since $R\le 1$ and $p>2$.
By the H\"older inequality, one has
\[
\begin{split}
\int_{Q_r(t_0,x_0)} \mathcal{U}^{2q+p + \frac{4(q+1)}{N}}\,dt\,dx
&\leq 
\int_{t_0-r^p}^{t_0}\left(\int_{Q_r(x_0)} \mathcal{U}^{(2q+p)\frac{2^*}{2}}\,dx\right)^{\frac{2}{2^*}}\left( \int_{Q_r(x_0)} \mathcal{U}^{2(q+1)}\,dx\right)^{\frac{2}{N}}\,dt\\
&\leq \left(\sup_{\tau\in (t_0-r^p, t_0)}\int_{\{\tau\}\times Q_r(x_0)} \mathcal{U}^{2(q+1)}\,dx\right)^{\frac{2}{N}}\int_{t_0-r^p}^{t_0}\left(\int_{Q_r(x_0)} \mathcal{U}^{(2q+p)\frac{2^*}{2}}\,dx\right)^{\frac{2}{2^*}}\,dt.
\end{split}
\]
Using \eqref{eq1551}, this implies
\begin{equation}\label{base-bis}
\iint_{Q_r(t_0,x_0)} \mathcal{U}^{2q+p + \frac{4(q+1)}{N}}\,dt\,dx
 \leq \left(C\frac{q^6}{(R-r)^p}\,\iint_{Q_R(t_0,x_0)} \left(1+\mathcal{U}^{2q+p}\right)\,dt\,dx\right)^{\frac{2}{N}+1}. 
\end{equation}
{\bf Step 5: the local $L^\infty$ estimate on \(\nabla u\).}
Take $q=q_j=2^{j+1}-1$ with \(j\in \mathbb{N}\). 
We set
\[
\gamma_j:=2\,q_j+p=2^{j+2}-2+p,\qquad \delta_j:=2\,q_j+p+\frac{4}{N}\,(q_j+1)=2^{j+2}-2+p+\frac{4}{N}\,2^{j+1},
\]
and 
\[
\tau_j=\frac{\delta_j-\gamma_j}{\delta_j-\gamma_{j-1}}\,\frac{\gamma_{j-1}}{\gamma_j}.
\]
We observe that \(\gamma_{j-1}<\gamma_j<\delta_j\) and \(\tau_j\in (0,1)\) is defined in a such a way that
\[
\frac{1}{\gamma_j} = \frac{\tau_j}{\gamma_{j-1}} + \frac{1-\tau_j}{\delta_j}.
\]
The estimate \eqref{base-bis} can be rewritten as
\[
\iint_{Q_r(t_0,x_0)} \mathcal{U}^{\delta_j}\,dt\,dx
 \leq \left(C\frac{q_j^6}{(R-r)^p}\,\iint_{Q_R(t_0,x_0)} \left(1+\mathcal{U}^{\gamma_j}\right)\,dt\,dx\right)^{\frac{2}{N}+1}
\]
By interpolation in $L^p$ spaces, we obtain
\[
 \iint_{Q_r(t_0,x_0)} \mathcal{U}^{\gamma_j}\,dt\,dx\le\left( \iint_{Q_r(t_0,x_0)} \mathcal{U}^{\gamma_{j-1}}\,dt\,dx\right)^{\tau_j\,\frac{\gamma_j}{\gamma_{j-1}}}\,\left(\iint_{Q_r(t_0,x_0)} \mathcal{U}^{\delta_j}\,dt\,dx\right)^{(1-\tau_j)\,\frac{\gamma_j}{\delta_j}}.
\]
By lengthy but elementary computations, we see that
\[
(1-\tau_j)\,\frac{\gamma_j}{\delta_j}=\frac{N}{N+4}\qquad \mbox{ and }\qquad \tau_j\,\frac{\gamma_j}{\gamma_{j-1}}=\frac{4}{N+4},
\]
thus the combination of the two previous inequalities leads to
\[
\iint_{Q_r(t_0,x_0)} \mathcal{U}^{\gamma_j}\,dt\,dx\le\left(\iint_{Q_r(t_0,x_0)} \mathcal{U}^{\gamma_{j-1}}\,dt\,dx\right)^\frac{4}{N+4}\,\left(C\frac{q^6_j}{(R-r)^p}\,\iint_{Q_R(t_0,x_0)} \Big(1+\mathcal{U}^{\gamma_j}\Big)\,dt\,dx\right)^\frac{N+2}{N+4}.
\]
We now use the Young inequality
\[
\begin{split}
\left(\iint_{Q_R(t_0,x_0)} \mathcal{U}^{\gamma_{j-1}}\,dt\,dx\right)^\frac{4}{N+4}&\left(C\,\frac{q^6_j}{(R-r)^p}\,\iint_{Q_R(t_0,x_0)} \Big(1+\mathcal{U}^{\gamma_j}\Big)\,dt\,dx\right)^\frac{N+2}{N+4}\\
&\le \frac{N+2}{N+4}\,\iint_{Q_R(t_0,x_0)} \Big(1+\mathcal{U}^{\gamma_{j}}\Big)\,dt\,dx\\
&+\frac{2}{N+4}\,\left(C\,\frac{q^6_j}{(R-r)^p}\right)^\frac{N+2}{2}\,\left(\iint_{Q_R(t_0,x_0)} \mathcal{U}^{\gamma_{j-1}}\,dt\,dx\right)^2.
\end{split}
\]
Thus we have proved the estimate
\begin{equation}\label{nevica_0}
\begin{split}
 \iint_{Q_r(t_0,x_0)} \mathcal{U}^{\gamma_j}\,dt\,dx&\le \frac{N+2}{N+4}\, \iint_{Q_R(t_0,x_0)} \Big(1+\mathcal{U}^{\gamma_j}\Big)\,dt\,dx\\
&+\frac{2}{N+4}\,\left(C\frac{q^6_j}{(R-r)^p}\right)^\frac{N+2}{2}\,\left( \iint_{Q_R(t_0,x_0)} \mathcal{U}^{\gamma_{j-1}}\,dt\,dx\right)^2.
\end{split}
\end{equation}
Next, we observe that
\[
\frac{|Q_r(t_0,x_0)|}{|Q_R(t_0,x_0)|^2}\leq C\, \left(\frac{r}{R}\right)^{N+p} \frac{1}{R^{N+p}} \leq \frac{C}{(R-r)^{N+p}}\leq \frac{C}{(R-r)^{p\,\frac{N+2}{2}}},
\]
where the last inequality relies on the two facts: $R\le 1$ and $p>2$. 
Hence, using that \(q_j\geq 1\), one gets
\[
\begin{split}
|Q_r(t_0,x_0)| &\leq C\,\left(\frac{q^6_j}{(R-r)^p}\right)^\frac{N+2}{2}\,|Q_R(t_0,x_0)|^2\\
& \leq C\,\left(\frac{q^6_j}{(R-r)^p}\right)^\frac{N+2}{2}\,\left(\iint_{Q_R(t_0,x_0)}\Big(1+ \mathcal{U}^{\gamma_{j-1}}\Big)\,dt\,dx\right)^2.
\end{split}
\]
By summing $|Q_r(t_0,x_0)|$ on both sides of \eqref{nevica_0}, this implies
\[
\begin{split}
 \iint_{Q_r(t_0,x_0)} \Big(1+\mathcal{U}^{\gamma_j}\Big)\,dt\,dx&\le \frac{N+2}{N+4}\, \iint_{Q_R(t_0,x_0)} \Big(1+\mathcal{U}^{\gamma_j}\Big)\,dt\,dx\\
&+C\,\left(\frac{q^6_j}{(R-r)^p}\right)^\frac{N+2}{2}\,\left( \iint_{Q_R(t_0,x_0)} \Big(1+\mathcal{U}^{\gamma_{j-1}}\Big)\,dt\,dx\right)^2.
\end{split}
\]
We can now appeal to \cite[Lemma 6.1]{Gi} and absorb the term on the right-hand side containing $1+\mathcal{U}^{\gamma_j}$, in a standard way. By using the definition of \(q_j\), this leads to
\begin{equation}
\label{conj}
\iint_{Q_r(t_0,x_0)} \Big(1+\mathcal{U}^{\gamma_j}\Big)\,dt\,dx\le C\,\frac{2^{3(N+2)j}}{(R-r)^{p\,\frac{N+2}{2}}}\,\left(\iint_{Q_R(t_0,x_0)}\Big(1+\mathcal{U}^{\gamma_{j-1}}\Big)\,dt\,dx\right)^2.
\end{equation}
We want to iterate the previous estimate on a sequence of shrinking parabolic cylinders. We fix two radii $0<r<R\le 1$, then we consider the sequence 
\[
R_j=r+\frac{R-r}{2^{j-1}},\qquad j\in\mathbb{N}\setminus\{0\},
\]
and we apply \eqref{conj} with \(R_{j+1}<R_j\) in place of \(r<R\). We introduce the notation
\[
Y_j=\iint_{Q_{R_{j}}(t_0,x_0)} \left(1+\mathcal{U}^{\gamma_{j-1}}\right)\,dt\,dx.
\]
Thus we get
\[
Y_{j+1} \le C\,2^{2\,p\,(N+2)\,j}\,(R-r)^{-\frac{p}{2}\,(N+2)}\,Y_{j}^2.
\]
By iterating the previous estimate starting from $j=1$, we obtain for every \(n\in \mathbb{N}\setminus \{0\}\),
\[
\begin{split}
Y_{n+1}&\le \Big(C\,2^{2\,p\,(N+2)}\,(R-r)^{-\frac{p}{2}\,(N+2)}\Big)^{\sum\limits_{j=0}^{n-1}(n-j)\,2^j}\,Y_1^{2^n}.
\end{split}
\]
possibly for a different constant \(C>1\).
We now take the power $2^{-n}$ on both sides:
\[
Y_{n+1}^{2^{-n}}\le \Big(C\,2^{2\,p\,(N+2)}\,(R-r)^{-\frac{p}{2}\,(N+2)}\Big)^{\sum\limits_{j=0}^{n-1}(n-j)\,2^{j-n}}
\,Y_1\leq C\,(R-r)^{-p\,(N+2)}
\,Y_1,
\]
where the last inequality relies on the fact that 
\[
\sum\limits_{j=0}^{n-1}(n-j)\,2^{j-n}\leq \sum\limits_{j=1}^{\infty}\frac{j}{2^{j}}=2.
\] 
We thus get
\begin{equation}
\label{eq1667}
\begin{split}
\|\mathcal{U}\|_{L^{\infty}(Q_r(t_0,x_0))}&=\lim_{n\to \infty}\left(\iint_{Q_{R_{n+1}(t_0,x_0)}} \mathcal{U}^{\gamma_{n}}\,dt\,dx\right)^\frac{1}{\gamma_{n}}\leq \limsup_{n\to +\infty}\left(Y_{n+1}^{2^{-n}} \right)^{\frac{2^n}{\gamma_n}}\\
&\leq \limsup_{n\to +\infty}\left( C\,(R-r)^{-p\,(N+2)}\,Y_{1} \right)^{\frac{2^n}{\gamma_n}}\\
& \leq C\,  (R-r)^{-p\,\frac{N+2}{4}}\,\left(\iint_{Q_R(t_0,x_0)} \Big(1+\mathcal{U}^{p+2}\Big)\,dt\,dx\right)^\frac{1}{4}.
\end{split}
\end{equation}
Here, we have also used that \(\gamma_0=p+2\) and $\gamma_n\sim 2^{n+2}$, for $n$ going to $\infty$.
Finally, in order to remove the dependence on the $L^{p+2}$ norm of the gradient, we use a standard interpolation trick. We write
\[
\begin{split}
\left(\iint_{Q_R(t_0,x_0)} \Big(1+\mathcal{U}^{p+2}\Big)\,dt\,dx\right)^\frac{1}{4} 
&\leq \left(\|\mathcal{U}\|_{L^{\infty}(Q_R(t_0,x_0))}^2\iint_{Q_R(t_0,x_0)} \mathcal{U}^{p}\,dt\,dx +|Q_R(t_0,x_0)|\right)^\frac{1}{4} \\
&\leq \|\mathcal{U}\|_{L^{\infty}(Q_R(t_0,x_0))}^{\frac{1}{2}}\,\left(\iint_{Q_R(t_0,x_0)} \mathcal{U}^{p}\,dt\,dx\right)^\frac{1}{4}  + C\,R^\frac{N+p}{4}.
\end{split}
\]
Inserting this estimate into  \eqref{eq1667}, we get 
\[
\begin{split}
\|\mathcal{U}\|_{L^{\infty}(Q_r(t_0,x_0))}
& \leq \frac{C}{(R-r)^{p\,\frac{N+2}{4}}}\,\left[\|\mathcal{U}\|_{L^{\infty}(Q_R(t_0,x_0))}^{\frac{1}{2}}\,\left(\iint_{Q_R(t_0,x_0)} \mathcal{U}^{p}\,dt\,dx\right)^\frac{1}{4}  + R^\frac{N+p}{4}\right]\\
&\leq \frac{1}{2}\,\|\mathcal{U}\|_{L^{\infty}(Q_R(t_0,x_0))} +  \frac{C}{(R-r)^{p\,\frac{N+2}{2}}}\,\left(\iint_{Q_R(t_0,x_0)} \mathcal{U}^{p}\,dt\,dx\right)^\frac{1}{2}  + \frac{C\,R^\frac{N+p}{4}}{(R-r)^{p\,\frac{N+2}{4}}}\\
&\leq \frac{1}{2}\,\|\mathcal{U}\|_{L^{\infty}(Q_R(t_0,x_0))} +  \frac{C}{(R-r)^{p\,\frac{N+2}{2}}}\,\left[\left(\iint_{Q_R(t_0,x_0)} \mathcal{U}^{p}\,dt\,dx\right)^\frac{1}{2}  +1\right], 
\end{split}
\]
where in the last line, we have used that \(R\leq 1\).
By \cite[Lemma 6.1]{Gi} again, we get
\[
\|\mathcal{U}\|_{L^{\infty}(Q_r(t_0,x_0))}
 \leq \frac{C}{(R-r)^{p\,\frac{N+2}{2}}}\,\left[\left(\iint_{Q_R(t_0,x_0)} \mathcal{U}^{p}\,dt\,dx\right)^\frac{1}{2}  + 1\right].
\]
Since \(\mathcal{U}\leq |\nabla u|\leq \sqrt{N}\, \mathcal{U}\), this  completes the proof. 
\end{proof}
\begin{rem}
An inspection of the proof reveals that the exponent $\alpha$ in \eqref{uepsbound} can be taken to be 
\[
\alpha=\left\{\begin{array}{lr}
\dfrac{N+2}{2},& \mbox{ if } N\geq 3,\\
\mbox{ any number \(>2\)}, & \mbox{ if } N=2.
\end{array}
\right.
\]
In the case $N=2$, the constant $C$ blows-up as $\alpha$ tends to $2$.
\end{rem}
\section{Proof of the Main Theorem}
\label{sec:5}

We take \(u\in L^{p}_{\rm loc}(I ; W^{1,p}_{\rm loc}(\Omega))\) to be a local weak solution  of equation \eqref{equationL}, i.e. it satisfies
\begin{equation}
\label{weak-ter}
\begin{split}
-\iint_{I\times\Omega} u\,\varphi_t\,dt\,dx&+\iint_{I\times\Omega} \langle \nabla F_{0}(\nabla u), \nabla \varphi\rangle\,dt\,dx=0,
\end{split}
\end{equation}
for every $\varphi\in C^\infty_0(I\times\Omega)$. We recall that $F_0$ indicates the convex function
\[
F_0(\xi)=\frac{1}{p}\,\sum_{i=1}^{N}|\xi_i|^{p},\qquad \mbox{ for every }\xi\in\mathbb{R}^N.
\]
Our program is as follows: 
\begin{itemize}
\item we approximate $u$ by solutions $u_\varepsilon$ of the regularized equation;
\vskip.2cm
\item we transfer the Lipschitz estimate of Proposition \ref{lm-apriori-estimate} from $u_\varepsilon$ to $u$;
\vskip.2cm
\item we use a scaling argument to ``rectify'' the local estimate and obtain \eqref{stimapriori}. 
\end{itemize}
We start by recalling that $u$ has the following additional properties: for every subinterval $J\Subset I$ and every open set $\mathcal{O}\Subset \Omega$, we have
\begin{equation}
\label{aggiuntive}
u_t\in L^{p'}(J; W^{-1,p'}(\mathcal{O}))\qquad \mbox{ and }\qquad u\in C(J;L^2(\mathcal{O})).
\end{equation}
Here \(W^{-1,p'}(\mathcal{O})\) is the topological dual space of \(W^{1,p}_0(\mathcal{O})\) and the latter is the completion of $C^\infty_0(\mathcal{O})$ with respect to the $L^p$ norm of the gradient.
\par 
We briefly recall the argument to get \eqref{aggiuntive}, for completeness.
Fix \(J\) and $\mathcal{O}$ as above, 
by using equation \eqref{weak-ter} and the fact that \(\nabla F_0(\nabla u)\in L^{p'}_{\rm loc}(I\times\Omega)\), we get
\[
\begin{split}
\left|\iint_{J\times\mathcal{O}} u\,\psi_t\,dt\,dx\right|&=\left|\iint_{J\times\mathcal{O}} \langle \nabla F_{0}(\nabla u), \nabla \psi\rangle\,dt\,dx\right|\\
&\le C\,\|\nabla \psi\|_{L^p(J\times\mathcal{O})}= \|\psi\|_{L^p(J;W^{1,p}_0(\mathcal{O}))}, \qquad \mbox{ for every }\psi\in C^\infty_0(J\times\mathcal{O}).
\end{split}
\]
By density, we can extend the linear functional
\[
\Lambda:\psi\mapsto \iint_{J\times\mathcal{O}} u\,\psi_t\,dt\,dx,
\]
to the whole space \(L^{p}(J; W^{1,p}_0(\mathcal{O}))\). This implies that (see for example \cite[Theorem 1.5, Chapter III]{Sh})
\[
\Lambda \in \Big(L^{p}(J; W^{1,p}_0(\mathcal{O}))\Big)^*=L^{p'}(J; W^{-1,p'}(\mathcal{O})).
\]
By definition of $\Lambda$ and of weak derivative, we get the first property in \eqref{aggiuntive}.
\par
The second property in \eqref{aggiuntive} follows by recalling that (see \cite[Proposition 1.2, Chapter III]{Sh})
\[
\Big\{\varphi\in L^p(J;W^{1,p}_0(\mathcal{O}))\, :\, \varphi_t\in L^{p'}(J;W^{-1,p'}(\mathcal{O}))\Big\}\subset C(\overline{J};L^2(\Omega)).
\]
In light of \eqref{aggiuntive} and 
since we are only interested in a local result, it is not restrictive to assume from the beginning that
\[
u\in L^{p}(I ; W^{1,p}(\Omega))\cap C(\overline{I};L^2(\Omega)) \qquad \mbox{ and }\qquad u_t \in L^{p'}(I; W^{-1,p'}(\Omega)).
\] 
{\bf Part 1: convergence of the approximation scheme}.
We remember the definition \eqref{Fe} of $F_\varepsilon$ for every \(\varepsilon\geq 0\).
We then consider the approximating initial boundary value problem parametrized by \(\varepsilon>0\) 
\[
\left\{\begin{array}{ccll}
v_t&=&\mathrm{div} \nabla F_\varepsilon(\nabla v),& \mbox{ in }I\times\Omega,\\
v&=&u,& \mbox{ on } I\times\partial \Omega,\\
v(0,\cdot)&=&u(0,\cdot), & \mbox{ in }\Omega.
\end{array}
\right.
\]
By \cite[Propositon 4.1, Chapter III]{Sh}, there exists a unique weak solution \(u_{\varepsilon}\in L^{p}(I ; W^{1,p}(\Omega))\) to this problem, such that 
\[
\left(u_{\varepsilon}\right)_t \in L^{p'}(I;W^{-1,p'}(\Omega))\qquad \mbox { and thus  } u_{\varepsilon}\in C(\overline{I};L^{2}(\Omega)).
\]
The boundary value condition is taken in the sense that 
\[
u_\varepsilon-u\in L^{p}(I;W^{1,p}_0(\Omega)),
\]
and the initial condition is taken in the $L^2$ sense, which is feasible thanks to the continuity properties of both $u_\varepsilon$ and $u$. The function $u_\varepsilon$ verifies
\begin{equation}
\label{weak-ter-approx}
-\iint_{I\times\Omega} u_\varepsilon\,\varphi_t\,dt\,dx+\iint_{I\times\Omega} \langle \nabla F_{\varepsilon}(\nabla u_\varepsilon), \nabla \varphi\rangle \,dt\,dx=0,
\end{equation}
for every $\varphi\in C^\infty_0(I\times\Omega)$. An integration by parts in \eqref{weak-ter-approx} leads to
\begin{equation}
\label{eq1783}
\int_I \left( \left(u_{\varepsilon}\right)_t,\varphi\right)_{(W^{-1,p' }, W^{1,p}_0)}\,dt+\iint_{I\times\Omega} \langle \nabla F_{\varepsilon}(\nabla u_{\varepsilon}),\nabla\varphi\rangle \,dt\,dx=0.
\end{equation}
By density, the above identity remains true for every \(\varphi\in L^{p}(I; W^{1,p}_0(\Omega))\). Then the choice  \(\varphi=u_\varepsilon-u\) yields
\[
\int_{I} \left( \left(u_{\varepsilon}\right)_t, u_{\varepsilon}-u\right)_{(W^{-1,p' }, W^{1,p}_0)}\,dt+\iint_{I\times\Omega} \langle \nabla F_{\varepsilon}(\nabla u_{\varepsilon}), \nabla u_{\varepsilon} -\nabla u \rangle \,dt\,dx=0.
\] 
By recalling the expression \eqref{Fe} of $F_\varepsilon$, we write
\[
\begin{split}
\langle \nabla F_{\varepsilon}\left(\nabla u_{\varepsilon}\right) -\nabla F_{0}\left(\nabla u\right), \nabla u_{\varepsilon} -\nabla u \rangle
&=\langle \nabla F_{0}\left(\nabla u_{\varepsilon}\right) -\nabla F_{0}\left(\nabla u\right), \nabla u_{\varepsilon} -\nabla u\rangle\\
&+\varepsilon\, \langle \nabla G\left(\nabla u_{\varepsilon}\right), \nabla u_{\varepsilon} -\nabla u\rangle.
\end{split}
\]
Thus the previous integral identity can be rewritten as
\[
\begin{split}
\int_{I} \left( \left(u_{\varepsilon}\right)_t, u_{\varepsilon}-u\right)_{(W^{-1,p' }, W^{1,p}_0)}\,dt&+\iint_{I\times\Omega} \langle \nabla F_0(\nabla u_{\varepsilon}), \nabla u_{\varepsilon} -\nabla u \rangle \,dt\,dx\\
&+\varepsilon\,\iint_{I\times \Omega} \langle \nabla G\left(\nabla u_{\varepsilon}\right), \nabla u_{\varepsilon} -\nabla u\rangle\,dt\,dx=0.
\end{split}
\]
Starting from \eqref{weak-ter}, we have similarly
\[
\int_{I} \left( u_t, u_{\varepsilon}-u \right)_{(W^{-1,p' }, W^{1,p}_0)}\,dt+\iint_{I\times\Omega} \langle \nabla F_{0}\left(\nabla u\right), \nabla u_{\varepsilon} -\nabla u \rangle \,dt\,dx=0.
\]
Upon substracting the two identities above, we get
\begin{equation}
\label{sottratte}
\begin{split}
\iint_{I\times\Omega} \left(  \left(u_{\varepsilon}\right)_t-u_t, u_{\varepsilon}-u \right)_{(W^{-1,p' }, W^{1,p}_0)}dt&+\iint_{I\times\Omega} \langle \nabla F_{\varepsilon}\left(\nabla u_{\varepsilon}\right) -\nabla F_{0}\left(\nabla u\right), \nabla  u_{\varepsilon} -\nabla u \rangle \,dt\,dx\\
&+\varepsilon\,\iint_{I\times \Omega} \langle \nabla G\left(\nabla u_{\varepsilon}\right), \nabla u_{\varepsilon} -\nabla u\rangle\,dt\,dx=0.
\end{split}
\end{equation}
For the first term, we rely on 
\[
\int_{I} \left(  \left(u_{\varepsilon}\right)_t-u_t, u_{\varepsilon}-u \right)_{(W^{-1,p' }, W^{1,p}_0)}\,dt
=\frac{1}{2} \int_{\{T_1\}\times\Omega}|u_{\varepsilon} - u|^2\,dx,
\]
which follows from the fact that 
\[
t\mapsto \frac{1}{2}\,\int_\Omega |u_\varepsilon(t,x)-u(t,x)|^2\,dx,
\]
is absolutely continuous on $I$, with derivative given exactly by 
\[
\Big(\left(u_{\varepsilon}\right)_t-u_t, u_{\varepsilon}-u \Big)_{(W^{-1,p' }, W^{1,p}_0)},\qquad \mbox{ for a.\,e. }t\in I,
\]
see \cite[Proposition 1.2]{Sh}.
\par
For the second term in \eqref{sottratte}, we can use that for every \(\xi, \zeta \in \mathbb{R}^N\), we have
\[
\begin{split}
\langle \nabla F_0(\xi) - \nabla F_0(\zeta), \xi-\zeta \rangle&=\sum_{i=1}^N \big(|\xi_i|^{p-2}\,\xi_i-|\zeta_i|^{p-2}\,\zeta_i\big)\,\big(\xi_i-\zeta_i\big)\\
&\ge 2^{2-p}\,\sum_{i=1}^N |\xi_i-\zeta_i|^p \geq \frac{1}{C}\, |\xi-\zeta|^p,
\end{split}
\]
for some \(C=C(N,p)>0\). The first inequality can be found in \cite[Section 10]{Lin}. On the other hand, the convexity of \(G\) implies that $\nabla G$ is a monotone map and thus
\[
\langle \nabla G\left(\nabla u_\varepsilon\right), \nabla u_{\varepsilon} - \nabla u  \rangle \geq \langle \nabla G(\nabla u), \nabla u_{\varepsilon} -\nabla u \rangle.
\]
By using these two pointwise estimates in \eqref{sottratte}, we thus get
\[
\frac{1}{2}\, \int_{\{T_1\}\times\Omega}|u_{\varepsilon} - u|^2\,dx
+\frac{1}{C}\, \iint_{I\times\Omega}\left|\nabla u_{\varepsilon} -\nabla u \right|^p\,dt\,dx
\leq \varepsilon \iint_{I\times\Omega} \left| \langle \nabla G\left(\nabla u\right), \nabla u_{\varepsilon} -\nabla u\rangle\right|\,dt\,dx.
\]
By the Cauchy-Schwarz inequality and the inequality \(|\nabla G(\xi)|\leq (1+ |\xi|)^{p-1}\), the right-hand side is not larger than
\[
\varepsilon \iint_{I\times\Omega} (1+|\nabla u|)^{p-1}\left| \nabla u_{\varepsilon} -u\right|\,dt\,dx,
\]
which, by the  Young inequality, in turn can be bounded from above by
\[
\frac{p-1}{p}\,\varepsilon \iint_{I\times\Omega}\left|\nabla u_{\varepsilon} -\nabla u\right|^p\,dt\,dx + \frac{\varepsilon}{p}\,\iint_{I\times\Omega} (1+|\nabla u|)^p\,dt\,dx.
\]
Then for every  \(\varepsilon>0\), we have
\[
\frac{1}{2} \int_{\{T_1\}\times\Omega}|u_{\varepsilon} - u|^2\,dx
+\left(\frac{1}{C}-\frac{p-1}{p}\,\varepsilon\right) \iint_{I\times\Omega}\left|\nabla u_{\varepsilon} -\nabla u\right|^p\,dt\,dx
\leq \frac{\varepsilon}{p}\,\iint_{I\times\Omega} (1+|\nabla u|)^p\,dt\,dx.
\]
This estimate guarantees that the family \(\{\nabla u_{\varepsilon}\}_{\varepsilon>0}\) is strongly convergent to $\nabla u$ in  \(L^{p}(I\times\Omega)\). This is now sufficient to pass to the limit in the uniform Lipschitz estimate \eqref{uepsbound}, which will be still valid for $u$, as well. By a covering argument, this in turn implies
\[
\nabla u\in L^\infty_{\rm loc}(I\times\Omega),
\]
as claimed. 
\vskip.2cm\noindent
{\bf Part 2: scale invariant estimate}. We finally focus on obtaining the estimate \eqref{stimapriori}. We fix
\[
\mathcal{Q}_{\tau,R}(t_0,x_0)=(t_0-\tau,t_0)\times (x_0-R,x_0+R)^N\Subset I\times\Omega,
\]
as in the statement of the Main Theorem. 
By recalling Remark \ref{rem:scaling}, we know that for every $R>0$ and $\mu>0$ the function 
\[
U_R(t,x)=\mu\,u(t_0+\mu^{p-2}\,R^p\,t,x_0+R\,x),
\]
is still a local weak solution of our parabolic equation, this time in the rescaled set
\[
\widetilde{I}\times \widetilde{\Omega}:=\left(\frac{T_0-t_0}{\mu^{p-2}\,R^p},\frac{T_1-t_0}{\mu^{p-2}\,R^p}\right)\times\frac{\Omega-x_0}{R}.
\]
We consider the compactly contained cube \(\mathcal{Q}_{\frac{\tau}{\mu^{p-2}R^p},1}(0,0)\Subset \widetilde{I}\times \widetilde{\Omega}\) and make the choice
\[
\mu^{p-2}\,R^p=\tau, \qquad \mbox{ that is, }\qquad \mu=\left(\frac{\tau}{R^p}\right)^\frac{1}{p-2},
\]
so that 
\[
\mathcal{Q}_{\frac{\tau}{\mu^{p-2}R^p},1}(0,0)=\mathcal{Q}_{1,1}(0,0)=Q_{1}(0,0).
\]
From {\bf Part 1}, we know that we can use the a priori estimate (\ref{uepsbound}) for $U_R$ 
on the anisotropic parabolic cubes  
\[
Q_{\frac{1}{2}}(0,0)=\left(-\frac{1}{2^p},0\right)\times\left(-\frac{1}{2},\frac{1}{2}\right)^N\qquad \mbox{ and }\qquad Q_1(0,0)=(-1,0)\times(-1,1)^N.
\]
This gives
\[
\|\nabla U_R\|_{L^{\infty}(Q_\frac{1}{2}(0,0))} \leq C\,\left(\iint_{Q_1(0,0)} |\nabla U_R|^p\,dt\,dx\right)^\frac{1}{2}  + C.
\]
If we scale back, we get 
\begin{equation}
\label{basis}
\|\nabla u\|_{L^\infty(\mathcal{Q}_{\frac{\tau}{2^p},\frac{R}{2}}(t_0,x_0))}\le C\,\left(\frac{\tau}{R^2}\right)^\frac{1}{2}\,\left(\fint_{\mathcal{Q}_{\tau,R}(t_0,x_0)} |\nabla u|^p\,dt\,dx\right)^\frac{1}{2}+C\,\left(\frac{R^2}{\tau}\right)^\frac{1}{p-2}.
\end{equation}
This already proves the claimed a priori estimate \eqref{stimapriori} for $0<\sigma\le 1/2^p$. 
\par
For any $\sigma\in (0,1)$, we  can take $(t_1,x_1)\in \mathcal{Q}_{\sigma\tau,\sigma R}(t_0,x_0)$ such that
\[
\|\nabla u\|_{L^\infty(\mathcal{Q}_{\sigma \tau,\sigma R}(t_0,x_0))}\leq 
\|\nabla u\|_{L^\infty(\mathcal{Q}_{\frac{1-\sigma}{2^p}\tau,\frac{1-\sigma}{2}R}(t_1,x_1))}.
\] 
Observe that for every $(s,y)\in \mathcal{Q}_{\sigma \tau,\sigma R}(t_0,x_0)$, we have
\[
\mathcal{Q}_{\frac{1-\sigma}{2^p}\tau,\frac{1-\sigma}{2}R}(s,y)\subset \mathcal{Q}_{(1-\sigma)\,\tau,(1-\sigma)\,R}(s,y) \Subset \mathcal{Q}_{\tau,R}(t_0,x_0).
\]
By applying \eqref{basis} on this parabolic cylinder, we get
\[
\begin{split}
\|\nabla u\|_{L^\infty(\mathcal{Q}_{\sigma\tau,\sigma R}(t_0,x_0))}&\leq 
\|\nabla u\|_{L^\infty(\mathcal{Q}_{\frac{1-\sigma}{2^p}\tau,\frac{1-\sigma}{2}R}(t_1,x_1))}\\
&\le \frac{C}{\sqrt{1-\sigma}}\,\left(\frac{\tau}{R^2}\right)^\frac{1}{2}\,\left(\fint_{\mathcal{Q}_{(1-\sigma)\tau,(1-\sigma)R}(t_1,x_1)} |\nabla u|^p\,dt\,dx\right)^\frac{1}{2}\\
&+C\,(1-\sigma)^\frac{1}{p-2}\,\left(\frac{R^2}{\tau}\right)^\frac{1}{p-2}.
\end{split}
\]
By observing that 
\[
\begin{split}
\fint_{\mathcal{Q}_{(1-\sigma)\tau,(1-\sigma)R}(t_1,x_1)} |\nabla u|^p\,dt\,dx&=\frac{1}{(1-\sigma)^{N+1}\,(2\,R)^N\,\tau}\,\int_{\mathcal{Q}_{(1-\sigma)\tau,(1-\sigma)R}(t_1,x_1)} |\nabla u|^p\,dt\,dx\\
&\le \frac{1}{(1-\sigma)^{N+1}\,(2\,R)^N\,\tau}\,\int_{\mathcal{Q}_{\tau,R}(t_0,x_0)} |\nabla u|^p\,dt\,dx\\
&= \frac{1}{(1-\sigma)^{N+1}}\,\fint_{\mathcal{Q}_{\tau,R}(t_0,x_0)} |\nabla u|^p\,dt\,dx,
\end{split}
\]
we eventually reach the desired conclusion.

\end{document}